\documentclass{amsart}
\usepackage{amsmath,amsthm}
\usepackage[english]{babel}
\usepackage{mathrsfs,amssymb}
\usepackage[matrix,arrow]{xy}
\usepackage{mathptmx}
\usepackage{setspace}
\usepackage{amscd}

 \theoremstyle{plain}    
 \newtheorem{theorem}{Theorem}[section]
 \numberwithin{equation}{section} %% Comment out for sequentially-numbered
 \numberwithin{figure}{section} %% Comment out for sequentially-numbered
 \theoremstyle{definition}
 \newtheorem{definition}[theorem]{Definition}
 
 \theoremstyle{plain}    
 \newtheorem{proposition}[theorem]{Proposition}
 \theoremstyle{remark}
 \newtheorem{remark}[theorem]{Remark}
 \newtheorem{example}[theorem]{Example}

\theoremstyle{plain}    
 \newtheorem{lemma}[theorem]{Lemma} %%Delete [thm] to re-start numbering
 
 \theoremstyle{definition}
 \theoremstyle{plain}    
 \newtheorem{corollary}[theorem]{Corollary} %%Delete [thm] to re-start numbering

\newcommand{\dif}{\mathrm{d}}

\newcommand{\Z}{\mathbb{Z}}

\newcommand{\R}{\mathbb{R}}
\newcommand{\C}{\mathbb{C}}

\newcommand{\T}{\mathscr T}
\newcommand{\D}{\mathscr D}

\newcommand{\cl}{\mathfrak{c}}
 \DeclareMathOperator{\tr}{tr}
\DeclareMathOperator{\Tr}{Tr}

\DeclareMathOperator{\Hom}{Hom} 
 \DeclareMathOperator{\Ker}{Ker}

\DeclareMathOperator{\Cl}{Cl}

\DeclareMathOperator{\ind}{ind}

\DeclareMathOperator{\supp}{supp}
\DeclareMathOperator{\str}{str}

\usepackage{epsfig}

\begin{document}


\title{$L^2$-index formula for proper cocompact group actions}

\author{Hang Wang}
\address{1227I Stevenson Center, Department of Mathematics, Vanderbilt University, Nashville, TN, USA, 37240.}
\address[Current Address]{Room 143, Mathematical Sciences Center, Jinchunyuan West Building, Tsinghua University, Hai Dian District, Beijing, China 100084. }
\email{hwang@math.tsinghua.edu.cn, hang.wang@vanderbilt.edu.}

\begin{abstract}We study index theory of $G$-invariant elliptic pseudo-differential operators acting on a complete Riemannian manifold, where a unimodular, locally compact group $G$ acts properly, cocompactly and isometrically. An $L^2$-index formula is obtained using the heat kernel method.
\end{abstract}

\maketitle

{\em Mathematics Subject Classifacation (2010):}{19K56, 58J35, 58J40}


\tableofcontents{}

\section{Introduction.}

\subsection{Main result.}

Let $X$ be a complete Riemannian manifold acted on properly, cocompactly and isomertrically by a locally compact unimodular group $G$ and let $E$ be a $\Z/2\Z$-graded $G$-vector bundle over $X$. 
Let $$P=\begin{pmatrix}0 & P_0^{\ast} \\ P_0 & 0 \end{pmatrix}: L^2(X,E)\rightarrow L^2(X,E)$$ be a $0$-order properly supported elliptic pseudo-differential operator invariant under the group action. 
Such an operator has a real-valued $L^2$-index defined as the difference of the von Neumann traces of the projections onto the closed $G$-invariant subspaces $\Ker P_0, \Ker P_0^{\ast}$ of $L^2(X,E)$:
$$\ind P=\tr_G P_{\Ker P_0}-\tr_G P_{\Ker P_0^{\ast}}.$$

The paper is to prove that the $L^2$-index of $P$ is calculated by the following topological formula:
\begin{equation}\label{iformula} \ind P=\int_{TX}(c\circ\pi)\cdot(\hat{A}(X))^2 \mathrm{ch}(\sigma_P).\end{equation}
Here $c\in C^{\infty}_c(X)$ is a non-negative function satisfying $\displaystyle\int_Gc(g^{-1}x)\mathrm{d}g=1$ for all $x\in X$, and $\pi: TX\rightarrow X$ is the projection. 

\subsection{Remarks on the result.}

The formula (\ref{iformula}) generalizes the $L^2$-index formula for free cocompact group actions due to  Atiyah \cite{Atiyah:1976} and the $L^2$-index formula for homogeneous spaces of unimodular Lie groups due to Connes and Moscovici \cite{Connes:1982}. 
The study of $L^2$-indices in general has implications in other areas of mathematics. 
For example, the non-vanishing of the $L^2$-index for the signature operator on $X$ indicates the existence of $L^2$-harmonic forms on $X$. 
The $L^2$-index is of interest in the study of discrete series representations \cite{Connes:1982} and has been modified for use in a proof of the Novikov conjecture for hyperbolic groups \cite{Connes:1990ht}. 

Our index formula (\ref{iformula}) is analogous to the type II theory in von Neumann algebra. The key feature of a type II index theory is that the elliptic operators being investigated are no longer Fredholm, but using some techniques analogous to those used in type II von Neumann theory, say, by formulating some trace, one may obtain generalized Fredholm indices associated to the elliptic operators.  Refer to \cite{Roe:1988qy, Roe:1988uq} for another example of this type.  

When the orbit space $X/G$ is an orbifold, the $L^2$-index discussed in this paper is not the same as the index for $X/G$ as a compact orbifold  \cite{Mathai:2001, Mathai:10}. 
For example, Dirac operators on a good orbifold are Fredholm and have integer indices, reflecting the information of the orbit space, while the $L^2$-indices of the Dirac operators lifted to the universal cover of the orbifold are rational numbers by definition. The integer indices and the rational indices are different in general \cite{Farsi:1992}. They coincide on spacial cases, for example, when the orbit space is a smooth manifold \cite{Atiyah:1976}.
Another example is that when both $X$ and $G$ are compact, (\ref{iformula}) is the same as the Atiyah-Singer index formula for compact manifolds, regardless of the group $G$ \cite{Atiyah:1968pd}, while the index formula corresponding to the orbifold $X/G$ involves group action \cite{Mathai:10}. 
Our formula is expected to have interesting applications when the group is not compact. 

We also notice the existence of $L^2$-index formula (in some special cases) when $X/G$ is a noncompact orbifold but has finite volume \cite{Stern:1989}, where the analysis on the strata of  of $X/G$ is heavily used. It is interesting to study the $L^2$-index (if exists) where the quotient is noncompact. However, our operator algebraic approach in finding the formula of $L^2$-index does not work for the case of  noncompact quotient. 
The reason is that when $G$ acts properly, cocompactly and isometrically on $X$, the group $G$ and the manifold $X$ are coarse equivalence, then we may use $G$, more precisely, $C^{\ast}(G)$ to study the elliptic operators on $X$ invariant under the action of $G$ \cite{Miscenko:1980fu, Lott:1992}.
However, when $X/G$ is not compact, the group $G$ has nothing to do with the $L^2$-index for $G$-invariant elliptic operators on $X$.  

Finally, (\ref{iformula}) fits into the framework of the higher index formula taking values in cyclic theory. In \cite{Perrot:2009}, a general formula was proved and the indices of Dirac operators take values in the entire cyclic homology of some subalgebra of the group $C^{\ast}$-algebra $C^{\ast}(G).$ Formally, for Dirac operators, (\ref{iformula}) is obtained from \cite{Perrot:2009} corollary 1.2 by taking $g\in G$ to be the group identity and by taking $n=0.$ We would like to have a deeper investigation on the connection of the two results in future.

\subsection{Idea of the proof.}

To prove (\ref{iformula}), regard $P$ as an element in the $K$-homology group $K_G^0(C_0(X)),$ from which $P$ has a higher index in $K_0(C^{\ast}(G)),$ where $C^{\ast}(G)$ is the maximal group $C^{\ast}$-algebra.
The $L^2$-index of $P$ depends only on the equivalence class of its higher index in $K_0(C^{\ast}(G))$. 
This is proved in section \ref{L^2K} by defining a trace on a dense holomorphic closed ideal $\mathcal{S}(\mathcal{E})$ in $\mathcal{K}(\mathcal{E})$, where $\mathcal{E}$ is a Hilbert $C^{\ast}(G)$-module having the same $K$-theory as $C^{\ast}(G)$. The trace is the von Neumann trace of a type II von Neumann algebra in the sense of Breuer \cite{Breuer:1968}.  A comprehensive discussion on the link between the $L^2$-index and the higher index may be found in \cite{Schick:2005}.

Secondly, in section \ref{5}, we reduce the problem of finding $\ind P$ into finding $\ind D$ for some Dirac type operator $D$, which has the same higher index as $P$. Kasparov's K-theoretic index formula \cite{Kasparov:1983} is essential in the argument. The formulation of Dirac type operators out of elliptic operators is also related to the vector bundle modification construction in the definition of geometric $K$-homology \cite{BD:1982}. 

 The final step is to calculate $\ind D$ using the heat kernel method. 
When $D$ is a first order $G$-invariant operator of Dirac type on $X$, we have the McKean-Singer formula for the $L^2$-index: 
\begin{equation}\label{ch1MS}\ind D=\tr_G e^{-tD^{\ast}D}-\tr_G e^{-tDD^{\ast}}, t>0.\end{equation}
In the case of a compact manifold without group action, a cohomological formula was obtained by studying the local invariants of metrics and connections \cite{ABP:1973, Gilkey:1973}. The proof of the local index formula was simplified by a rescaling argument of Getzler \cite{Getzler:1986} on the asymptotic expansion of the heat kernel $e^{-tD^2}$ around $t=0$.
Since the index $\ind D$ in (\ref{ch1MS}) is local when $t\to0+$, 
the group action does not affect the calculation. 
The proof is based on a modification of the proofs in \cite{Roe:1998ad, BGV} and is complete in section \ref{6}.

\subsection{Acknowledgement.}
The work is modified after my PHD thesis and is funded by NSF, Vanderbilt University and IHES.  
Special thanks go to Professor Gennadi Kasparov for proposing this topic and for his advice. 
After writing up my paper, I received many comments as well as warm helps. I wish to express my sincere gratitude to all the professors who have helped me.
Finally, I would like to thank the referee for the helpful remarks. 

\section{Preliminaries.}

Let $G$ be a {locally compact} and {\em unimodular} group, that is, there is a bi-invariant Haar measure $\mu$ on $G$. For example, compact groups and discrete groups are unimodular. 
Set $\mathrm{d}g\doteq \dif\mu(g)$ and we have 
$$\mathrm{d}(tg)=\mathrm{d}g, \mathrm{d}(gt)=\mathrm{d}g\text{ and }\mathrm{d}(g^{-1})=\mathrm{d}g\text{ for any }g,t\in G.$$
Let $X$ be a complete Riemannian manifold, on which $G$ acts {properly}, {cocompactly} and { isometrically}, that is, the pre-image of any compact set under the continuous map 
$$G\times X\rightarrow X\times X: (g,x)\mapsto (g\cdot x,x)$$
 is compact, 
 the quotient space $X/G$ is compact, 
 and $G$ respects the metric $<\cdot, \cdot>$: 
 $$<x,y>=<gx,gy>\text{ for all }x, y\in X, g\in G.$$
The reason to consider proper cocompact actions is the existence of a cutoff function on $X$.

\begin{definition}
{A nonnegative function $c\in C_c^{\infty}(X)$ is a \emph{cutoff function} if for all $x\in X,$ $$\displaystyle\int_Gc(g^{-1}x)\dif g=1.$$}
\end{definition}

\begin{remark}
     A proper cocompact $G$-space has a cutoff function $c\in C_c^{\infty}(X)$ given by
     $$c(x)=\frac{h(x)}{\int_G h(g^{-1}x)\mathrm{d}g},$$ where $h(x)\in C^{\infty}_c(X)$ is nonnegative and has non-empty intersection with each orbit. 
\end{remark}

\begin{example}
{Let $G$ be a Lie group with a compact subgroup $H$, and let $X=G/H$ be the homogeneous space consisting of all the left cosets of $H$ in $G$. The action of $G$ on $X$ is proper. Further, let $E$ be a representation space of $H$. The induced representation $Y=G\times_HE$, which forms a $G$-vector bundle over $X$, is a proper $G$-space.  According to the slice theorem, every proper space has such a local structure.}
\end{example}

\begin{theorem}[Slice theorem]
Let $G$ be a locally compact group and $X$ be a proper $G$-space. Then for any $x\in X$ and for any neighborhood $O$ of $x$ in $X$, there exists a compact subgroup $K$ of $G$ with $G_x\doteq\{g\in G|gx=x\}\subset K$ and a $K$-slice $S$ such that $x\in S\subset O.$ Recall that A $K$-invariant subset $S\subset X$ is a $K$-slice in $X$ if 
\begin{enumerate}
\item The union $G(S)$(tubular set) of all orbits intersecting $S$ is open;
\item There is a $G$-equivariant map $f: G(S)\rightarrow G/K$ (the slicing map), such that $S=f^{-1}(eK)$.
\end{enumerate}  
\end{theorem}

An introduction to the slice theorem may be found in \cite{Antonyan:2010} section 2. 
According to \cite{Bredon} Ch.II Theorem 4.2, the tubular set $G(S)\subset X$ with a compact slicing subgroup $K$ is $G$-homeomorphic to $G\times_KS$. 

\begin{remark}
Since $X$ is covered by $G$-invariant neighborhoods and since $X/G$ is compact, then $X$ admits a finite sub-cover, that is,
\begin{equation}\label{local structure}X=\cup_{i=1}^N G\times_{K_i}S_i=\cup_{i=1}^N G(S_i).\end{equation}
The local structure (\ref{local structure}) of $X$ defines a $G$-invariant measure $\mathrm{d}x$ on $X$. In fact, The measure of a set in $G(S_i)$ is calculated from the measure on $G$ and on $S_i$ divided by the measure of $K_i$. 
Then the measure of a set $T\subset X$ is defined using a partition of unity argument. 
The $1$-density on the Riemannian manifold $X$ also defines the same measure.
\end{remark}

In order to introduce ellipticity, we recall the following definitions concerning pseudo-differential operators.
Let $(E, p)$ be a finite dimensional complex $G$-vector bundle over $X$, that is, there is a smooth $G$ action on $E$ such that $p(gv)=gp(v)$ for $v\in E$ and the maps of the fibers $g: E_x\rightarrow E_{gx}$ are linear. 
Let $\pi: T^{\ast}X\rightarrow X$ be the projection map and $\pi^{\ast}E$ over $T^{\ast}X$ be the pull-back bundle of $E$.
Here, $E=E_0\oplus E_1$ is  $\Z/2\Z$-graded and the $G$-action is grading preserving.
The $G$-actions on $E_0,E_1$ give rise to a $G$-bundle $\Hom(\pi^{\ast}E_0,\pi^{\ast}E_1)$ over $T^{\ast}X$.
A \emph{symbol} function $\sigma$ of order $m$ is a continuous section of this $G$-bundle satisfying 
\begin{equation} |\frac{\partial^a}{\partial x^{|a|}}\frac{\partial^b}{\partial\xi^{|b|}}\sigma(x,\xi)|\le C_{a,b,K}(1+\|\xi\|)^{m-|b|} \end{equation} for $x$ in any compact set $K\subset X$ and $\xi$ in the fiber $T_xX$, where $C_{a,b,K}$ is a constant depending on $a,b,K.$
Here $a=(a_1, \ldots, a_n), b=(b_1, \ldots, b_n)$ and $\displaystyle|a|=\sum_{i=1}^n a_i, |b|=\sum_{i=1}^n b_i  (\mathrm{dim}X=n).$
The set of all order $m$ symbols is denoted by $S^m(X; E_0, E_1)$ and a \emph{principal symbol} of order $m$ is an element in the quotient $S^{m}(X;E_0,E_1)/S^{m-1}(X; E_0, E_1).$ We shall omit the word ``principal" from now on.

Each symbol $\sigma$ has an \emph{amplitude} $p$ defined by 
$$p(x,y,\xi)=\alpha(x,y)\sigma(q(y,(x,\xi_x))),$$
where $\alpha\in C^{\infty}(X\times X)$ has support contained in a small neighborhood of the diagonal so that $\alpha(x,x)=1$ and $\alpha(x,y)\ge0$ for all $x,y\in X$ and $q: X\times T^{\ast}X\rightarrow T^{\ast}X: (y,(x,\xi_x))\mapsto (y,\xi_y)$ ($\xi_y$ is the parallel transport of $\xi_x$ from $x$ to $y$). 
Conversely, $$\sigma(x,\xi)=p(x,x,\xi).$$

Denote by $C^{\infty}_c(X,E)$ the set of smooth sections of $E$ with compact support in $X$ and $G$ acts on $C^{\infty}_c(X,E)$ by $(g\cdot f)(x)=g(f(g^{-1}x)),$ for all $g\in G, f\in C_c^{\infty}(X,E).$To each amplitude $p(x,y,\xi)$, we may construct a \emph{pseudo-differential operator} $P_0: C^{\infty}_c(X, E_0)\rightarrow C^{\infty}(X,E_1)$ by 
\begin{equation}\label{constrop}
P_0u(x)=\int_{X\times T_x^{\ast}X}e^{i\Phi(x,y,\xi)}p(x, y, \xi)u(y)\dif y\dif\xi_x,
\end{equation}
where $\Phi(x,y,\xi)=<\exp^{-1}_x(y), \xi_x>$ is the phase function.
The \emph{Schwartz kernel} $K_{P_0}(x,y)\in \Hom({E_0}_y,{E_1}_x)$ of $P_0$, that is,  
$$\displaystyle P_0u(x)=\int_X K_{P_0}(x,y)u(y)\dif y\text{ for all }u(x)\in C^{\infty}_c(X,E_0),$$ is expressed in the following distributional sense,
\begin{equation}K_P(x,y)(w)=\int_{X\times T^{\ast}X}e^{i\Phi(x,y,\xi)}p(x,y,\xi)w(x,y)\dif x\dif y\dif\xi,\quad w\in C^{\infty}_c(X\times X).\end{equation}

We assume $P_0$ to be {\em $G$-invariant}, that is, $$\displaystyle P_0(gf)=gP_0(f), f\in C_c^{\infty}(X, E_0),\text{ for all }g\in G.$$
Clearly, the Schwartz kernel of a $G$-invariant operator $P_0$ satisfies that
\begin{equation}K_{P_0}(x,y)=K_{P_0}(gx,gy) \text{ for all } x,y\in X, g\in G.\label{kernel}\end{equation} 
In addition, assume $P_0$ to be \emph{properly supported}, that is, for any compact subset $K\subset X$, the subsets $\supp K_P\cap(K\times X)$ and $\supp K_P\cap(X\times K)$ in $X\times X$ are compact. 
Proper supportness of $P_0$ in particular implies that $P_0$ maps $C^{\infty}_c(X, E_1)$ to itself.

Choose a $G$-invariant Hermitian structure on $E$ and let $L^2(X,E)$ be the completion of $C_c(X,E)$ under inner product, $\displaystyle <f, g>_{L^2}=\int_X<f(x), g(x)>_{E_x}\dif x.$ 
Let $P$ be an essentially self-adjoint operator on $L^2(X,E)$ with odd grading, in the form of 
$$P=\begin{pmatrix}0 & P_0^{\ast} \\ P_0 & 0\end{pmatrix}.$$
Without loss of generality, $P$ is assumed to be of order $0$ and then $P$ extends to be a bounded self-adjoint operator on $L^2(X,E).$

We shall use the following notations and we omit $E, F$ or $X$ when it is clear in the context.
\begin{itemize}
\item $\Psi^n(X;E,F)$: the set of order $n$ pseudo-differential operators from $C_c^{\infty}(X,E)$ to $C^{\infty}(X,F)$;
\item $\Psi^n_{G}(X; E,F)$: the subset of $G$-invariant elements in $\Psi^n(X;E,F)$;
\item $\Psi^n_{G,p}(X;E,F)$: the subset of properly supported elements in $\Psi^n_{G}(X; E,F)$;
\item $\Psi^n_c(X;E,F)$: the subset of $\Psi^n(X;E,F)$ having compactly supported Schwartz kernels. 
\end{itemize}

The symbol of an operator $P\in\Psi^{\ast}_{G,p}$ is $G$-invariant.
Conversely, if $\sigma(x,\xi)$ is a $G$-invariant symbol, then there is an operator in $\Psi^{\ast}_{G,p}$ with symbol $\sigma(x,\xi)$. To do this we construct $P$ using (\ref{constrop}) and use the averaging operation from \cite{Connes:1982}:
\begin{equation}\mathrm{Av}_G: \Psi^{\ast}_c\rightarrow\Psi^{\ast}_{G,p}: P\mapsto\int_G gPg^{-1}\dif g. \end{equation}
Then $Av_G(cP)\in\Psi^{\ast}_{G,p}$, where $c$ is a cutoff function for $X$, has the symbol $\sigma(x,\xi)$. 
%Note that $\text{If }P\in\Psi_{G,p}^{\ast}\text{ then }P=\mathrm{Av}_G(cP).$

\begin{definition}\label{defell} \cite{Kasparov:2008}
A pseudo-differntial operator $P\in\Psi^{m}(X;E,F)$ is \emph{elliptic} if there exists $Q\in\Psi^{-m}(X; F, E)$ so that
\begin{equation}\|\sigma_P(x,\xi)\sigma_Q(x,\xi)-I\|\to0\text{ and }\|\sigma_Q(x,\xi)\sigma_P(x,\xi)-I\|\to0 \label{elli}\end{equation} 
uniformly in $x\in K$  as $\xi\to\infty$ in $T_{x}^{\ast}X$ for any compact subset $K$ in $X$.
Without loss of generality, we will consider order-$0$ elliptic pseudo-differential operators  $P_0\in\Psi^0_{G,p}(X; E_0, E_1)$ with the condition (\ref{elli}) replaced by
\begin{equation}\|\sigma_{P_0}(x,\xi)\sigma_{P_0^{\ast}}(x,\xi)-I\|\to0\text{ and }\|\sigma_{P_0^{\ast}}(x,\xi)\sigma_{P_0}(x,\xi)-I\|\to0.\label{ell}\end{equation}
\end{definition}

\begin{proposition} \label{PDOprop}
\begin{enumerate}
\item If $P\in\Psi_c^n$, then $\mathrm{Av}_G(P)\in\Psi_{G,p}^n$.
\item If $P\in\Psi^n_{G,p}(X)$ is elliptic, then there exists a parametrix $Q\in\Psi^{-n}_{G,p}(X)$ such that 
\begin{equation}1-PQ=S_1\in\Psi^{-\infty}_{G,p}(X), 1-QP=S_2\in\Psi^{-\infty}_{G,p}(X),\end{equation}
where $\Psi^{-\infty}_{G,p}(X)=\cap_{n\in\R}\Psi^n_{G,p}(X)$ is the set of smoothing operators. 
\item If $S\in\Psi^{-\infty}_{G,p}(X)$, then $K_S(\cdot,\cdot)$ is smooth and properly supported.
\end{enumerate}
\end{proposition}

\begin{proof}
(1) Clearly, $\mathrm{Av}_G(P)\in\Psi_{G,p}^{\ast}(X).$ 
If $p(x,y,\xi)\in S^m(X\times T^{\ast}X)$ is the amplitude then $P\in\Psi^n_c$ implies that $K=\{(x,y)\in X\times X| p(x,y,\xi)\neq0\}$ is compact. 
Using the fact that the Riemannian metric on $T^{\ast}X$ is $G$-invariant and the measure on $X$ is $G$-invariant, we calculate the amplitude for $\mathrm{Av}_G(P)$ as $$\int_G p(g^{-1}x,g^{-1}y,\xi_{g^{-1}x})\dif g$$ which is of order $n$ because the integral is taken over a set $\displaystyle\{g\in G| (g^{-1}x, g^{-1}y)\in K \}$ which is compact.

(2) Let $P\in\Psi_{G,p}^n(X)$ be elliptic and $c\in C_c^{\infty}(X)$ be a cutoff function for $X$. Cover $X$ by finitely many bounded open balls $\{U_i\}_{i=1}^N$ such that  $\supp(c)\subset \cup_{i=1}^N U_i.$  Let $\{a_i\}_{i=1}^N$ be a partition of unity subordinate to the finite cover. 
Since $P$ is elliptic, which implies that for any compact $K\subset X$, there exists a constant $C_K$ such that $|\sigma_P|\ge C(1+|\xi|)^n$ uniformly for all $|\xi|\ge C_K$, 
then there exist $Q_i\in\Psi_c^{-n}(U_i), 1\le i\le N$ so that 
$$PQ_i-a_i=R_{1,i}, Q_iP-a_i=R_{2,i}$$ 
are elements in $\Psi_c^{-\infty}(U_i)$.
Extend the elements in $\Psi_c^{\ast}(U_i)$ to $\Psi_c^{\ast}(X)$ and then 
$$c\sum_{i=1}^NQ_iP-c=c\sum_{i=1}^NR_{2,i}.$$
Since $\displaystyle\sum_{i=1}^NQ_i\in\Psi_c^{-n}(X)$, $\displaystyle\sum_{i=1}^NR_{2,i}\in\Psi_c^{-\infty}(X)$, 
we set $\displaystyle Q=\int_Gg(c\sum_{i=1}^N Q_i)\dif g\in\Psi_{G,p}^{-n}(X)\text{ and }S=\int_Gg(c\sum_{i=1}^N R_{2,i})\dif g\in\Psi_{G,p}^{-\infty}(X).$
Then 
\begin{align*} 
QP&=\int_Gg(c\sum_{i=1}^NQ_i)P\dif g=\int_G g(c)g(\sum_{i=1}^NQ_iP)\dif g\\
&=\int_G g(c) \dif g+\int_G g(c)g(\sum_{i=1}^nR_{2,i})\dif g=I+S.
\end{align*}
Similarly, there is a $\displaystyle Q'=\int_Gg(\sum_{i=1}^N Q_ic)\dif g\in\Psi_{G,p}^{-n}(X)$ and $S'\in\Psi_{G,p}^{-\infty}(X)$ so that $PQ'-I=S'.$ 
Since $Q'+SQ'-Q=(1+S)Q'-Q=Q(PQ'-1)=QS'$, then $Q'-Q\in\Psi^{-\infty}_{G,p}(X)$. 
Hence there are $S_1, S_2=S\in\Psi^{-\infty}_{G,p}(X)$ such that $PQ=1+S_1, QP=1+S_2.$

(3) If $S\in\Psi^{-\infty}_{G,p}(X)$, then $cS\in\Psi^{-\infty}_c(X)$. 

We know that $cS\in\Psi^{-\infty}_c(X)$ is equivalent to the fact that $K_{cS}(x,y)$ is smooth and compactly supported in $X\times X$. Therefore the statement follows from the fact that $$\displaystyle K_S(x,y)=K_{\mathrm{Av}_G(cS)}(x,y)=\int_GK_{cS}(g^{-1}x,g^{-1}y)\dif g$$ and the fact that the integral vanishes outside a compact set in $G$.
\end{proof}

\section{The $G$-trace and the $L^2$-index.}
 
When $X$ is compact and when $G$ is trivial, the dimensions of $\Ker P_0$ and $\Ker P_0^{\ast}$ are finite and their difference defines the index of $P$. 
In our case we measure the size of $\Ker P_0$ or $\Ker P_0^{\ast}$ by a real number in terms of von Neumann dimension. 
An $L^2$-index of $P$, analogous to the Fredholm index is defined, motivated by the $L^2$-index defined by Atiyah \cite{Atiyah:1976} and modified upon \cite{Connes:1982}.

\subsection{The $G$-trace of operators on $X$.}

Recall that a bounded operator $T$ on a Hilbert space $H$ is of {\em trace class} if $\displaystyle\sum_{i=1}^{\infty}|<|T|e_i, e_i>|<\infty$, where $\{e_i\}_{i=1}^{\infty}$ is an orthonormal basis of the Hilbert space and its {\em trace} calculated by 
$$\displaystyle\tr(T)=\sum_{i=1}^{\infty}<Te_i, e_i>$$ 
is independent of the orthonormal basis.

\begin{definition}\label{Gtrace}
A bounded operator $S: L^2(X, E)\rightarrow L^2(X,E)$, which commutes with the action of $G$, is of {\em $G$-trace class} if $\phi |S|\psi$ is of trace class for all $\phi, \psi\in C^{\infty}_c(X)$. 

If $S$ is a $G$-trace class operator, we calculate the $G$-trace by the formula 
\begin{equation}\tr_G(S)=\tr(c_1Sc_2),\label{deftrace}\end{equation} where $c_1, c_2\in C_c^{\infty}(X)$ are nonnegative, satisfying $c_1c_2=c$ for some cutoff function $c$ on $X.$ 
\end{definition}

\begin{remark}
When $G$ is discrete, Definition \ref{Gtrace} is essentially the definition of the $G$-trace class operator appearing in \cite{Atiyah:1976}. 
Similarly to Lemma 4.9 of \cite{Atiyah:1976}, we prove in the following proposition that $\tr_G$ is well defined, that is, $\tr_G$ is independent of the choice of $c_1, c_2$ and $c.$ 
\end{remark}

\begin{proposition}
Let $S$ (bounded, $G$-invariant and positive) be a $G$-trace class operator and $c_1, c_2, d_1, d_2\in C^{\infty}_c(X)$ be nonnegative functions satisfying $\displaystyle\int_G c_1(g^{-1}x)c_2(g^{-1}x) \dif g=1$ and $\displaystyle\int_G d_1(g^{-1}x)d_2(g^{-1}x) \dif g=1$, which means that $c=c_1c_2$ and $d=d_1d_2$ are cutoff functions on $X.$ Then $\tr(c_1Sc_2)=\tr(d_1Sd_2).$
\end{proposition}

\begin{proof}
Let $K=\{g\in G|\supp(g\cdot (d_1d_2))\cap\supp c\neq\emptyset\}$ and then $K$ is compact by the properness of the group action. Hence,
\begin{align*}
\tr(c_1Sc_2)&=\tr(\int_G[g\cdot(d_1d_2)]c_1Sc_2\dif g)=\tr(\int_K[g\cdot(d_1d_2)]c_1Sc_2\dif g)\\
&=\int_K\tr([g\cdot d_1] [g\cdot d_2] c_1 Sc_2)\dif g=\int_K\tr(c_1 [g\cdot d_1]D [g\cdot d_2] c_2)\dif g\\
&=\int_K\tr([g^{-1}\cdot c_1] d_1 Sd_2 [g^{-1}\cdot c_2])\dif g=\tr([\int_Gg(c_1c_2)\dif g] d_1Sd_2)=\tr(d_1Dd_2).
\end{align*}
\end{proof}

Using the fact that $\tr$ is a well-defined trace on compactly supported operators on $X$, it is easy to see that $\tr_G$ is linear, faithful, normal and semi-finite. The tracial property of $\tr_G$ is proved in the following proposition together with some other properties of $\tr_G$.

\begin{proposition}\label{propositions of trace}
\begin{enumerate}
\item A properly supported smoothing operator $A\in\Psi_{G,p}^{-\infty}$ is of $G$-trace class. If $K_A: X\times X\rightarrow \Hom E$ is the kernel of $A$, then its $G$-trace is calculated by 
\begin{equation}\tr_G(A)=\int_X c(x)\Tr K_A(x,x)\dif x, \end{equation}
where $c$ is a cutoff function and $\Tr$ is the matrix trace of $\Hom E$. In fact, this formula holds for all $G$-invariant operators having smooth integral kernel.
\item If $A\in\Psi^{\ast}_{G}$ is of $G$-trace class, so is $A^{\ast}.$
\item If $A\in\Psi^{\ast}_{G}$ is of $G$-trace class and $B\in\Psi^{\ast}_{G}$ is bounded, then $AB$ and $BA$ are of $G$-trace class. 
\item If $AB$ and $BA$ are of $G$-trace class, then $\tr_G(AB)=\tr_G(BA)$ 
\end{enumerate}
\end{proposition}

\begin{proof}
Let $\phi,\psi\in C^{\infty}_c(X)$ and let $\{\alpha_i^2\}_{i=1}^N$ be the $G$-invariant partition of unity in Proposition \ref{PDOprop}.

(1) Proposition \ref{PDOprop} (3) shows that $A\in\Psi_{G,p}^{-\infty}$ has smooth kernel.
Then $K_{\phi A\psi}(x,y)=\phi(x)K_A(x,y)\psi(y),$ is smooth and compactly supported, which means that $\phi A\psi$ is of trace class.
The integral formula for smoothing operators is classical. A proof may be found at \cite{Shubin} Section 2.21. 

(2) Because $\bar{\psi}A\bar{\phi}$ has finite trace by definition, then $\phi A^{\ast}\psi=(\bar{\psi}A\bar{\phi})^{\ast}$ is of trace class.

(3) Assume we have a $G$-trace class operator $A\in\Psi^{\ast}_{G,p}$. Since $\supp\psi$ is compact then as $A$ is properly supported, there is a compact set $K$ so that $\supp A\psi\subset K$. Choose $\eta, \zeta\in C_c^{\infty}(X)$ with $K\subset\supp\eta$ and $\eta\zeta=\eta$. 
Then $\eta A\psi=A\psi$ and for a bounded $B\in\Psi^{\ast}_G$, we have that $\phi BA\psi=\phi B\zeta\eta A\psi=(\phi B\zeta)(\eta A\psi)$. 
Since $\phi B\zeta$ is bounded operator with compact support and $\eta A\psi$ is trace class operator then their product is also a trace-class operator. So $BA$ is of $G$-trace class. $AB$ is of $G$-trace class because $B^{\ast}A^{\ast}$ is of $G$-trace class.

If $A\in\Psi^{\ast}_G$, then we have $A=A_1+A_2$ so that $A_1\in\Psi^{\ast}_{G,p}$ and $A_2$ has smooth kernel (which follows from a classical statement saying that the Schwartz kernel is smooth off the diagonal). 
Then the statement follows from the fact that $\phi A_1\psi$ has  smooth, compactly supported Schwartz kernel.

(4) We first prove a special case when $AB$ and $BA$ have smooth integral kernels.
Use the slice theorem (\ref{local structure}) to get $\displaystyle\{G\times_{K_i}S_i=G(S_i)\}_{i=1}^N$, $G$-invariant tubular open sets covering $X$. Then there exist $G$-invariant maps $\alpha_i: X\rightarrow [0,1]$ with $\supp\alpha_i\subset G(S_i)$ such that $\displaystyle\sum_{i=1}^N\alpha_i^2=1$. In fact, let $\tilde{\alpha_i}^2$ be a partition of unity of $X/G$ subordinate to the open sets $G(S_i)/G$. Lift $\tilde{\alpha_i}$ to $\alpha_i$ on $X$, then $\{\alpha_i^2\}$ is a $G$-invariant partition of unity of $X.$
Then:
  \begin{align*}  &\tr_G(AB) =\int_X\int_X c(x)\Tr(K_A(x,y)K_B(y,x))\mathrm{d}y\mathrm{d}x \\
      =& \sum_{i,j}\int_{G\times_{K_i}S_i}\int_{G\times_{K_j}S_j}\alpha_i^2(x)\alpha_j^2(y)c(x)\Tr(K_A(x, y)K_B(y, x))\mathrm{d}y\mathrm{d}x\\
      =&\sum_{i,j}\frac{1}{\mu(K_i)\mu(K_j)}\int_{S_i}\int_{S_j}\alpha_i^2(\bar{s})\alpha_j^2(\bar{t})\int_G\int_Gc(\overline{ht})\Tr(K_A(\overline{gs},\overline{ht})K_B(\overline{ht},\overline{gs}))\mathrm{d}g\mathrm{d}s\mathrm{d}h\mathrm{d}t\\
      =&\sum_{i,j}\frac{1}{\mu(K_i)\mu(K_j)}\int_{S_i}\int_{S_j}\alpha_i^2(\bar{s})\alpha_j^2(\bar{t})\int_{G}\Tr(K_B(\overline{ht},\bar{s})K_A(\bar{s},\overline{ht}))\mathrm{d}g\mathrm{d}s\mathrm{d}h\mathrm{d}t=\tr_G(BA).
  \end{align*}
Note that
in the third equality, $\bar{gs}\doteq (g,s)K_i=x\in G\times_{K_i}S_i$ and $\bar{ht}\doteq(h,t)K_j=y\in G\times_{K_j}S_j$ and by definition $\alpha_i(\bar{s})=\alpha_i(\bar{gs}), \alpha_j(\bar{t})=\alpha_j(\bar{ht})$. 
Also, we have used (\ref{kernel}), $\mathrm{d}h^{-1}=\mathrm{d}h, \mathrm{d}(h^{-1}g)=\mathrm{d}g,$ and change of variable in the fourth equality.

If either $A$ or $B$ are properly supported, (say $A$), then $\displaystyle\tr_G(AB)=\tr(c_1ABc_2)=\tr(\int_G c_1A g\cdot(c_1c_2) Bc_2)$. So the set $\{g\in G | c_1Ag\cdot c_1\neq0\}$ is compact in $K$, which allows us to interchange $\tr$ and $\displaystyle\int_K$, and to use tracial property of $\tr$ and $G$-invariance of $A$ and $B$ to prove $\tr_G(AB)=\tr_G(BA)$

In general let $A=A_1+A_2$ and $B=B_1+B_2$ where $A_1, B_1$ are properly supported and $A_2, B_2$ are bounded and have smooth kernel. 
Then $\tr_G(AB)=\tr_G(BA)$ using the special cases discussed above. 
\end{proof}

\begin{remark}
Let $S$ be a bounded $G$-invariant operator with smooth integral kernel and define $S_i\doteq\alpha_i S\alpha_i\in\Psi_c^{-\infty}(X;E,E)$. 
Then $\alpha_i^2S$ is of $G$-trace class by Proposition \ref{propositions of trace} (3). 
We may calculate $\tr_G(S)$ as follows,
\begin{align*}
\tr_G(S)=&\tr_G(\sum_{i=1}^N\alpha_i^2S)=\sum_{i=1}^N\int_{G\times_{K_i}S_i}\alpha_i(x) c(x)\Tr K_S(x,x)\alpha_i(x)\mathrm{d}x \\
=&\sum_{i=1}^N\int_{G\times_{K_i}S_i}c(x)\Tr K_{S_i}(x,x)\mathrm{d}x\\
=&\sum_{i=1}^N\mu(K_i)^{-1}\int_{G\times S_i}c((g,s))\Tr K_{S_i}((g,s),(g,s))\mathrm{d}g\mathrm{d}s \\
=&\sum_{i=1}^N\mu(K_i)^{-1}\int_{G\times S_i}c((g,s))\Tr K_{S_i}((e,s),(e,s))\mathrm{d}g\mathrm{d}s\\
=&\sum_{i=1}^N\mu(K_i)^{-1}\int_{S_i}\Tr K_{S_i}(s,s)\mathrm{d}s.
\end{align*}

The above trace formula  coincides with the trace formulas in the special cases. 
\begin{enumerate}
%\item Let $X$ be compact with $G=\{e\}$ and $H_0, H_1$ be separable Hilbert spaces, with orthonormal basis $\{e_i\}, \{f_i\}$. There is a faithful, normal and semi-finite trace $\Tr$ on $\mathcal{B}(H_0, H_1)$ so that $\Tr(A)=\sum(Ae_i, f_i)$ for $A\ge 0$. 
%The set of trace class operators is defined as the linear span of positive operators with finite trace. 
\item If the action is free and cocompact, then $X=G\times U$, and for a bounded positive self-adjoint operator $S$ with smooth kernel, we obtain $$\displaystyle\tr_G(S)=\int_U \Tr K_S(x,x)\mathrm{d}x.$$ 
\item For a homogeneous space of a Lie group $X=G/H$, and for $S\in\Psi^{-\infty}_{G,p}(X)$, we have $\tr_G(S)=K_S(e,e),$ where $e$ is the group identity. 
\end{enumerate}
\end{remark}

\begin{proposition}\label{fin ker}
If $P_0\in\Psi_{G,p}^m$ is an elliptic operator, then $P_{\Ker P_0}\in\Psi_G^{-\infty}$ is of $G$-trace class.\end{proposition}

\begin{proof}
By Proposition \ref{PDOprop}, there is a $Q\in\Psi_{G,p}^{-m}$ so that $1-QP_0=S\in\Psi^{-\infty}_{G,p}.$ 
Then apply it to $P_{\Ker P_0}$ and get $P_{\Ker P_0}=SP_{\Ker P_0}\in\Psi^{-\infty}_G$.
The statement is proved using (1) and (3) of Proposition \ref{propositions of trace}.
\end{proof}

\begin{remark} \label{trace intuition}
Let $\{\alpha_i^2\}_{i=1}^N$ be the $G$-invariant partition of unity in the proof of Proposition \ref{propositions of trace} (4). 
Then by the same property and for any bounded operator $T\in \Psi^{-\infty}_G$, we have
$$\tr_GT=\sum_{i=1}^N\tr_G(\alpha_iT\alpha_i)$$ 
where every summand $\alpha_iT\alpha_i$ is $G$-invariant and restricts to a slice 
$G\times_{K_i} S_i$ in $X.$ 

The action of $G$ on the vector bundle $E$ is induced by the action of its subgroup $K_i$ on $V\doteq E|_{S_i}$, the restriction of the bundle $E$ over a subset $\{(e,s)K_i| s\in S_i\}$ of $X.$
 \[ \begin{CD}
V=E|_{S_i} @>k\in K_i>> V=E|_{S_i} \\
@VVV              @VVV\\
\{e\}\times S_i       @>k\in K_i>>  \{e\}\times S_i
\end{CD} \]
Then we have the identification of the Hilbert spaces $L^2(G\times_{K_i} S_i, E)=(L^2(G)\otimes L^2(S_i, V))^{K_i}$, which consists of the elements of $L^2(G)\otimes L^2(S, V)$ invariant under the action of $K_i$, where $k\in K_i$ acts by 
$$k(f(g), h(s))=(f(gk^{-1}), k\cdot h(s)), g\in G, s\in S_i, f\in L^2(G), h\in L^2(S_i,V).$$
The $G$-invariance of $\ker P_0$ implies that $\alpha_iP_{\ker P_0}\alpha_i$ is an element of 
\begin{equation}\mathcal{R}(L^2(G))\otimes\mathcal{B}(L^2(S_i,V)),\label{decom}\end{equation}
and this element commutes with the action of the group $K_i$ on $\mathcal{R}(L^2(G))\otimes\mathcal{B}(L^2(S_i,V)).$  
Here $\mathcal{R}(L^2(G))$ is the weak closure of the right regular representation of $G$ ($L^1(G)$ more precisely) represented on $L^2(G).$
On this set there is a natural von Neumann trace determined by 
$$\tau(R(f)^{\ast}R(f))=\int_G|f(g)|^2\mathrm{d}g,$$ 
where 
$f\in L^2(G)\cap L^1(G)$
and 
$\displaystyle R(f)=\int_Gf(g)R(g)\mathrm{d}g$. 
Here $R(g)$ is the right regular representation of $g\in G$ on $L^2(G)$.
Also $\mathcal{B}(L^2(S_i,V))$ also has a subset where an operator trace $\tr$ can be defined.
There is a natural normal, semi-finite and faithful trace defined on $\mathcal{R}(L^2(G))\otimes\mathcal{B}(L^2(S_i,V))$ given by $\tau\otimes\tr$ on algebraic tensors. Refer to \cite{Perez:2008} Section 2 for a detailed description.

This trace coincides with the $G$-trace in Definition \ref{Gtrace} on the set of bounded $G$-invariant operators with smooth kernel. 
In fact, by a partition of unity argument, such an operator is finite sum of operators of form $S=A\otimes B\in \mathcal{R}(L^2(G))\otimes\mathcal{B}(L^2(S_i,V))$, which commutes with the action of $K_i$, where $A$ and $B$ have smooth kernel.
In \cite{Connes:1982}, it has been shown that $\tau(A)=K_A(e,e)$. Let $d\in C^{\infty}_c(G)$ be any cutoff function for $G.$ Then $\displaystyle\tau(A)=\int_G d(g)K_A(g,g)\mathrm{d}g.$ 
Hence, 
\begin{align*} 
\tau(A)\tr(B)=&\int_Gd(g)K_A(g,g)\mathrm{d}g\int_{S_i}\mathrm{Tr}K_B(s,s)\mathrm{d}s\\
                     &=\int_{G\times S_i}\frac{1}{\mu(K_i)}c((g,s))\mathrm{Tr}K_S((g,s),(g,s))\mathrm{d}g\mathrm{d}s\\
                     &=\int_{G\times_{K_i}S_i}c(x)\mathrm{Tr}K_S(x,x)\mathrm{d}x.
\end{align*}
Therefore we have proved the following proposition.
\end{remark}

\begin{proposition}\label{vndim}
On $\Psi^{-\infty}_{G,p}(X; E, E)$, the $G$-trace equals the natural von Neumann trace on the von Neumann algebra $\mathcal{R}(L^2(X,E))$, the weak closure of all the natural bounded operators on $L^2(X,E)$ which commute with the action of $G$.  
The $L^2$-index is the difference of the von Neumann trace of $P_{\Ker P_0}$ and $P_{\Ker P_0^{\ast}}.$
\end{proposition}

\begin{example}
When $G$ is a discrete group acting on itself by left translations, define $$c(g)=\begin{cases}1 & g=e \\ 0 & g\neq e \end{cases}, \text{ then}$$ 
$$\tr cT=\sum_{g\in G}<cT\delta_g,\delta_g>=<\sum_{g\in G}g^{-1}(cT)g\delta_e, \delta_e>=<\mathrm{Av}(cT)\delta_e, \delta_e>=\tr_G \mathrm{Av}(cT).$$
In general, $\mathrm{Av(c\cdot)}: \mathcal{B}(L^2(G))\rightarrow \mathcal{R}(L^2(G))$ extends the map 
$\Psi^{\ast}_c\rightarrow\Psi^{\ast}_{G,p}: cT\rightarrow \mathrm{Av}(cT)$ which preserves the corresponding trace. When $T$ is $G$-invariant, $T=\mathrm{Av}(cT)$ and then $\tr_G T=\tr_G\mathrm{Av}(cT)=\tr cT$ motivates the $\tr_G$ formula.
\end{example}

\subsection{The $L^2$-index of elliptic operators on $X$.}

According to Proposition \ref{fin ker}, we define a real valued \emph{$G$-dimension} of $K$, a closed $G$-invariant subspace of $L^2(X, E)$, by $$\dim_G K=\tr_G P_K$$ where $P_K$ is the projection from $L^2(X,E)$ onto $K$, and is $G$-invariant. 
\begin{definition}
The \emph{$L^2$-index} of the elliptic operator $P\in\Psi^{\ast}_{G,p}$ is 
\begin{equation}\ind P=\dim_G \Ker P_0-\dim_G \Ker P_0^{\ast}.\label{L2 ind}\end{equation}
\end{definition}
An immediate computation of the $L^2$-index is given by the following proposition,

\begin{proposition}\label{prMS}
Let $P\in\Psi_{G,p}^m$ be elliptic and $Q$ be an operator so that $1-QP_0=S_1, 1-P_0Q=S_2$ are of $G$-trace class, then $$\ind P=\tr_G S_1-\tr_G S_2.$$
\end{proposition}

\begin{proof}
The proof is similar to the one in \cite{Atiyah:1976}. We have 
$$S_1P_{\Ker P_0}=P_{\Ker P_0}\text{ and }P_{\Ker P_0^{\ast}}S_2=P_{\Ker P_0^{\ast}}$$ 
by composing $QP_0=1-S_1$ with $P_{\Ker P_0}$ and by composing $P_{\Ker P_0^{\ast}}$ with $1-S_2=P_0Q$ respectively.
Also, $P_0(QP_0)=(P_0Q)P_0$ implies that $P_0S_1=S_2P_0.$
Set $R=\delta_0(P_0^{\ast}P_0)P_0^{\ast}$ where $\delta_0(0)=1, \delta_0(x)=0\text{ for }x\ne0$, so 
$$RP_0=1-P_{\Ker P_0}, P_0R=1-P_{\Ker P_0^{\ast}}.$$
On one hand $\tr_G S_1-\tr_G P_{\Ker P_0}=\tr_GS_1(1-P_{\Ker P_0})=\tr_G(S_1RP_0)$. On the other hand $\tr_G S_2-\tr_G P_{\Ker P_0^{\ast}}=\tr_GS_2(1-P_{\Ker P_0^{\ast}})=\tr_G(S_2P_0R)=\tr_G(P_0S_1R)$.
Therefore $\tr_G S_1-\tr_G S_2=\tr_G P_{\Ker P_0}-\tr_G P_{\Ker P_0^{\ast}}$ by Proposition \ref{propositions of trace}.
\end{proof}

From the last proposition we derive the following \emph{McKean-Singer formula}.

\begin{corollary}\label{MS}
If $D=\begin{pmatrix}0 & D_0^{\ast} \\ D_0 & 0 \end{pmatrix}\in\Psi_G^1(X; E, E)$ is a first order essentially self-adjoint elliptic differential operator, then 
 \begin{equation}\ind D=\tr_G(e^{-tD_0^{\ast}D_0})-\tr_G(e^{-tD_0D_0^{\ast}})\text{ for all } t>0,\label{MS}\end{equation} 
 which in particular means that $\ind D$ is independent of $t>0.$
\end{corollary}

To prove (\ref{MS}) we need the following lemma.

\begin{lemma}\label{exp tr}
Let $D_0$ be as above, then $e^{-tD_0D_0^{\ast}}$ and $e^{-tD_0^{\ast}D_0}$ are of $G$-trace class. 
\end{lemma} 
 
\begin{proof}[Proof of lemma \ref{exp tr}]
It is sufficient to prove the case when $t=1.$ The proof is based on the ideas in \cite{Gilkey:1974, Connes:1982}. Also refer to the heat kernel estimate for a Riemannian manifold in \cite{Lesch:2009} Appendix B. 

If $\lambda\in\C-[0,\infty)$, then $\lambda I-D_0^{\ast}D_0$ is invertible. 
Let $L=\{\lambda\in\C| d(\lambda, \R_+)=1\}$ be clock-wise oriented. 
Then 

$$e^{-D_0^{\ast}D_0}=\frac{1}{2\pi i}\int_L\frac{e^{-\lambda}}{\lambda I-D_0^{\ast}D_0}\dif\lambda.$$

Let $\phi, \psi\in C^{\infty}_c(X)$ be supported in a compact set $K\subset X$ and let $\{\alpha_i\}_{i=1}^N$ be a partition of unity subordinated to an open cover of $K$ of local coordinate charts. 
We approximate $\phi e^{-D_0^{\ast}D_0}\psi$ by an operator in $\Psi^{-\infty}_c$ (with smooth and compactly supported Schwartz kernel) by inverting $\lambda I-D_0^{\ast}D_0$ ``locally". 

Let $p_i$ be the full symbol of $\alpha_i\phi(\lambda I-D_0^{\ast}D_0)^{-1}\psi$, having the asymptotic sum 
\begin{equation}p_i\sim\sum_{j=2}^{\infty} a_{-j}\text{ on a local coordinate, 
that is, }\mathrm{Op}(p_i-\sum_{j=2}^{m}a_{-j})\in\Psi_c^{-m-1},\text{ } \forall m>1,\label{asysum}\end{equation} where $\mathrm{Op}$ means the operator corresponding to the local symbol.

For any $l>0$ and $n>0$, choose a large enough $M$ and set the operator approximating $\alpha_i\phi(\lambda I-D_0^{\ast}D_0)^{-1}\psi$ to be 
\begin{equation}P_i(\lambda)=\mathrm{Op}(\sum_{j=2}^{M}a_{-j}),\label{psum}\end{equation} 
in the sense that $P_i(\lambda)$ is analytic in $\lambda$ and for any fixed $u\in L^2(X, E),$
\begin{equation}
\|(P_k(\lambda)-\alpha_i\phi(\lambda I-D_0^{\ast}D_0)^{-1}\psi)u\|_l \le C(1+|\lambda|)^{-n},\label{estimate}
\end{equation} where the norm is the Soblev $l$-norm $\|\cdot\|_l.$
The estimate (\ref{estimate}) is made possible by the asymptotic sum (\ref{asysum}).
In fact, let $r(x,\xi)$ be the symbol of $R\doteq P_i(\lambda)-\alpha_i\phi(\lambda I-D_0^{\ast}D_0)^{-1}\psi$ which is in $S^{-M-1}$ and then the left hand side of (\ref{estimate}) is $\|Ru\|_l=\int(1+|\xi|^2)^l|\widehat{Ru}(\xi)|\dif\xi$, where 
$Ru(x)=\int e^{<x-y,\xi>}r(x,\xi)u(y)\dif y\dif\xi$ can be controlled by the right hand side of (\ref{estimate}) when $M>>2l+2n.$ This is because by the definition of $r(x,\xi)$ there is a constant $C$ so that $|r(x,\xi)|<C(1+|\xi|)^{-M-1}$.

Set 
\begin{equation}E(\lambda)=\sum_{i=1}^N E_i(\lambda)=\sum_{i=1}^N\frac{1}{2\pi i}\int_L e^{-\lambda} P_i(\lambda)\dif\lambda.\label{E}\end{equation}
Then the following two observations prove that $\phi e^{-D_0^{\ast}D_0}\psi$ is of trace class. 

\begin{enumerate}
\item The operator $E(\lambda)$ is a compactly supported operator with smooth Schwartz kernel.
\begin{proof}[Proof of claim]
We need to show that the Schwartz kernel of $E_k(\lambda)$ is smooth.
In view of (\ref{psum}) and (\ref{E}), it is sufficient to show that $\mathrm{Op}(a_j), j\le -2$ has smooth kernel and $\displaystyle\int_Le^{-\lambda}\partial^{\beta}(\mathrm{Op}(a_j)u)\dif\lambda$ is integrable for all $\beta$.
This claim can be proved by the symbolic calculus (\cite{Gilkey:1974}). 
The crucial part in the argument is that by the symbolic calculus, all $a_j, j\le-2$ contain the factor $e^{-\sigma_2(D_0^{\ast}D_0)}$
and the fact that $e^{-t\sigma_2(D_0^{\ast}D_0)}$ is rapidly decreasing in $\xi$.
\end{proof}
\item The function $(E(\lambda)-\phi e^{-D_0^{\ast}D_0}\psi)u$ is in $H^l$ for any fixed $u\in L^2.$
\begin{proof}[Proof of claim]
Using (\ref{estimate}), and fixing a $u\in L^2(X,E)$

\begin{align*}\label{}
\|(E(\lambda)-\phi e^{-D_0^{\ast}D_0}\psi)u\|_l &\le\frac{1}{2\pi}\sum_{i=1}^N\int_L e^{-\lambda} \|(P_i(\lambda)-\alpha_i\phi(D_0^{\ast}D_0-\lambda I)^{-1}\psi)u\|_l\dif\lambda \\
&\le C\int_L e^{-\lambda}(1+|\lambda|)^{-n}\dif\lambda \to0 \text{ as } n\to\infty.
\end{align*}
\end{proof}
\end{enumerate}

Note that $E(\lambda)$ depends on the number $M$, which is chosen based on $l,n$, and it has a compactly supported smooth kernel by the first claim and hence $E(\lambda)u\in C_c^{\infty}\subset H^l.$ 
The second claim shows that $\phi e^{-D_0^{\ast}D_0}\psi$ is in $H^l$. 
(When $n\to\infty$, there is a sequence of $E(\lambda)\in H^l$ approaching $\phi e^{-D_0^{\ast}D_0}\psi$ in $\|\cdot\|_l$ norm.)

Let $l\to\infty$, then by the Sobolev Embedding Theorem $(\phi e^{-D_0^{\ast}D_0}\psi)u$ is smooth for all $u\in L^2$.
Therefore $\phi e^{-D_0^{\ast}D_0}\psi$ has a compactly supported smooth kernel and is a trace-class operator.
\end{proof}

\begin{proof}[Proof of Corollary \ref{MS}]
Let
$\displaystyle Q=\int_0^te^{-sD_0^{\ast}D_0}D_0^{\ast}\dif s$, 
which is the parametrix of $D_0$ because 
$$1-QD_0=e^{-tD_0^{\ast}D_0}, I-D_0Q=e^{-tD_0D_0^{\ast}}$$ which is of $G$-trace class by the lemma. 
The statement follows from Proposition \ref{prMS}.
The independence of $t$ can be carried out by a modification of the second proof of \cite{BGV} Theorem 3.50.
\end{proof}

\section{The connection of the $L^2$-index to the $K$-theoretic index.}\label{L^2K}

Let $f\in C_0(X)$ be identified as an operator on $L^2(X, E)$ by point-wise multiplication. 
Let $A\in\Psi^0_{p}(X; E, E)$ be elliptic in the sense of Definition \ref{defell}. 
Using the Rellich lemma one may check that $A_0: L^2(X, E_0)\rightarrow L^2(X, E_1)$ satisfies the following conditions: 
\begin{itemize}
\item $(A_0A_0^{\ast}-I)f\in \mathcal{K}(L^2(X,E_1)),$ $(A_0^{\ast}A_0-I)f\in \mathcal{K}(L^2(X,E_0)),$; 
\item $Af-fA \in \mathcal{K}(L^2(X,E))$;
\item $A_0-g\cdot A_0\in\mathcal{K}(L^2(X, E_1), L^2(X, E_2))$ for all $g\in G.$
\end{itemize}
Hence $A$ represents an element in the $K$-homology group $K^0_G(C_0(X))$.

Topologically, the \emph{$K$-theoretic index} of $[A]\in K^0_G(C_0(X))$, according to \cite{Kasparov:1983}, is defined by 
$$\mathrm{Ind}_t A\doteq[p]\otimes_{C^{\ast}(G,C_0(X))} j^G([A])\in K_0(C^{\ast}(G)),$$ 
which is the image of $[A]$ under the descent map 
$$j^G: KK^G(\C, C_0(X))\rightarrow KK(C^{\ast}(G), C^{\ast}(G, C_0(X)))$$ 
composed with the intersection product with $[p]\in KK(\C, C^{\ast}(G, C_0(X))),$
$$[p]\otimes_{C^{\ast}(G,C_0(X))}: KK(C^{\ast}(G,C_0(X)), C^{\ast}(G))\rightarrow KK(\C, C^{\ast}(G)).$$ 
Here $p\doteq(c\cdot g(c))^{\frac{1}{2}}$ an idempotent in $C_c(G, C_0(X))$,being the image of $1$ under the $\ast$-homomorphism $\C\rightarrow C^{\ast}(G, C_0(X))$ and defining an element in $K_0(C^{\ast}(G, C_0(X))).$

Analytically, the $K$-theoretic index of $A$ is constructed explicitly as follows \cite{Kasparov:1988dw}.
First of all, embed $C_c(X,E)$ in a larger Hilbert $C^{\ast}(G)$-module $C^{\ast}(G, L^2(X,E))$ and after completion under the norm of the Hilbert module, we obtain a $C^{\ast}(G)$-module $\mathcal{E}$ containing $C_c(X,E)$ as a dense subalgebra. Note that $\mathcal{E}$ is a direct summand of $C^{\ast}(G, L^2(X,E))$ and is obtained by compressing the $C^{\ast}(G)$-module $C^{\ast}(G, L^2(X,E))$ with the idempotent $p$.

Then the operator $$\bar A\doteq\mathrm{Av}(cA): C_c(X,E)\rightarrow C_c(X,E)$$ in $\Psi_{G,p}^0(X;E,E)$ extends to two bounded maps $\bar A: L^2(X,E)\rightarrow L^2(X,E)$ and $\bar A: \mathcal{E}\rightarrow\mathcal{E}$ with $\|\bar A\|_{\mathcal{E}}\le\|\bar A\|_{L^2(X,E)}.$ 
Denote by $\mathcal{B}(\mathcal{E})$ the $C^{\ast}$-algebra of all bounded operators on $\mathcal{E}$ having an adjoint and being $C^{\ast}(G)$-module maps.
Then $\bar A:\mathcal{E}\rightarrow\mathcal{E}$ defines an element in $\mathcal{B}(\mathcal{E})$ according to \cite{Kasparov:2008}.
On the Hilbert $C^{\ast}(G)$ module $\mathcal{E}$,  for $e, e_1, e_2\in C_c(X,E),$ a {\em rank one} operator is defined by 
$$\theta_{e_1,e_2}(e)(x)=e_1(e_2,e)(x)=\int_X(\int_G\theta_{g(e_1)(x),g(e_2)(y)}\dif g)e(y)\dif y, \forall x\in X.$$ 
The closure of the the linear combinations of the rank one operators under the norm of $\mathcal{B}(\mathcal{E})$ is the set of {\em compact operators}, denoted by $\mathcal{K}(\mathcal{E}).$ The elements of $\mathcal{K}(\mathcal{E})$ can be identified with the integral operators with $G$-invariant continuous kernel and with proper support.
The following proposition indicates some features of elements from $\mathcal{B}(\mathcal{E}), \mathcal{K}(\mathcal{E}).$
 
\begin{proposition}\cite{Kasparov:2008}\label{kascomp}
{If the symbol of the $G$-invariant properly supported operator $P$ of order $0$ is bounded in the cotangent direction by a constant, then the norm of $P$ in $\mathcal{B}(\mathcal{E})/\mathcal{K}(\mathcal{E})$ does not exceed that constant. The operator $P$ is compact i.e. $P\in\mathcal{K}(\mathcal{E})$, if the symbol of $P$ is $0$ at infinity (in the cotangent direction).}
\end{proposition} 

Since $\bar A$ is elliptic, which means that $\|\sigma_{\bar{A}}(x,\xi)^2-1\|\to 0 \text{ as } \xi\to 0, x\in K$
uniformly for any compact set $K\subset X$, 
then according to Proposition \ref{kascomp} we have that $\bar A^2-Id\in\mathcal{K}(\mathcal{E})$.
Let us set $\bar A=\begin{pmatrix}0&\bar{A_0}^{\ast}\\\bar{A_0}&0\end{pmatrix}$, then $[\bar{A_0}]\in K_1(\mathcal{B}(\mathcal{E})/\mathcal{K}(\mathcal{E}))$. 
The analytical \emph{$K$-theoretic index}, $\mathrm{Ind}_a A$, is image of this class in the $K$-theory of the quotient algebra under the boundary map 
$\partial: K_{\ast}(\mathcal{B}(\mathcal{E})/\mathcal{K}(\mathcal{E}))\rightarrow K_{\ast+1}(\mathcal{K}(\mathcal{E}))$ of the six term exact sequence associated to the short exact sequence $0\rightarrow\mathcal{K}(\mathcal{E})\rightarrow\mathcal{B}(\mathcal{E})\rightarrow\mathcal{B}(\mathcal{E})/\mathcal{K}(\mathcal{E})\rightarrow 0.$

\begin{remark}
The set of finite rank $\mathcal{K}(\mathcal{E})$-valued projections forms a finite generated projective $C^{\ast}(G)$-module. Then Theroem 3 of section 6 in \cite{Kasparov:1981} implies that $K_{\ast}(\mathcal{K}(\mathcal{E}))\simeq K_{\ast}(C^{\ast}(G))$. Hence, $\mathrm{Ind}_a P\in K_0(C^{\ast}(G)).$ 
\end{remark}

As a generalization of the Atiyah-Singer index theorem, Kasparov proved that $\mathrm{Ind}_a$ and $\mathrm{Ind}_t$ coincide \cite{Kasparov:1983, Kasparov:2008}. We will simply use $\mathrm{Ind}$ to denote the $K$-theoretic index. 
In summary, the $K$-theoretic index under the homomorphism 
$\mathrm{Ind}: K^0_G(C_0(X))\rightarrow KK(\C, C^{\ast}(G))\simeq K_0(\mathcal{K}(\mathcal{E}))$ 
is calculated by 
\begin{equation}[(L^2(X,E), A)]\mapsto [(\mathcal{E}, \bar A)]\mapsto [\begin{pmatrix} \bar A_0\bar A_0^{\ast} & \bar A_0\sqrt{1-\bar A_0^{\ast}\bar A_0} \\ \sqrt{1-\bar A_0^{\ast}\bar A_0}\bar A_0^{\ast} & 1-\bar A_0^{\ast}\bar A_0 \end{pmatrix}]-[\begin{pmatrix} 1 & 0\\ 0 &0 \end{pmatrix}].\label{KInd}\end{equation}
Note that the second arrow is the Fredholm picture of $KK(\C, C^{\ast}(G))$ via boundary map.
 
Given the $K$-theoretic index $\mathrm{Ind} A\in K_0(\mathcal{K}(\mathcal{E}))$, we will define the a homomorphism $K_0(\mathcal{K}(\mathcal{E}))\rightarrow\R.$ To do this we find a dense subalgebra $\mathcal{S}(\mathcal{E})$ of $\mathcal{K}(\mathcal{E})$ on which a ``trace" can be defined and which is closed under holomorphic functional calculus. 
Since $\mathcal{K}(\mathcal{E})$ is generated by $G$-invariant operators with continuous and properly supported kernel, we define $\mathcal{S}(\mathcal{E})$ to be the subset of the bounded $G$-invariant operators with smooth kernels.
Let $S: C_c^{\infty}(X,E)\rightarrow C_c^{\infty}(X,E)$ be a $G$-invariant smoothing operator. 
Extend $S$ to an operator $\bar{S}\in\mathcal{B}(\mathcal{E})$ and then $\bar S\in\mathcal{S}(\mathcal{E}).$ 
Define the trace on $\bar S \in\mathcal{S}(\mathcal{E})$ by $\tr_G(S)$ and still denote by $\tr_G$.
The trace is well defined for all the elements of $\mathcal{S}(\mathcal{E})$. An element of $\mathcal{S}(\mathcal{E})$ is viewed as matrices with $C^{\ast}(G)$-entries. The trace on such a matrix is the Breuer von Neumann trace \cite{Breuer:1968} on the image of the following map $\mathcal{S}(\mathcal{E})\mapsto\mathcal{S}(\mathcal{E}\otimes_{C^{\ast}(G)}\mathcal{R}(L^2(G)))\subset\mathcal{R}(L^2(X,E))$. Here $\mathcal{S}(\mathcal{E}\otimes_{C^{\ast}(G)}\mathcal{R}(L^2(G)))$ is a subset of all $G$-trace class operators and its elements are represented as matrices with $\mathcal{R}(L^2(G))$-entries.
Recall (Remark \ref{trace intuition}) that $\tr_G$ is defined on a dense subset of the $G$-invariant operators on $L^2(X, E)$,which can be represented as elements of
$$\mathcal{R}(L^2(G))\otimes(\oplus_{i,j}\mathcal{B}(L^2(U_i, E), L^2(U_j, E)),$$ 
and an element of this set can be expressed in terms of a $\mathcal{R}(L^2(G))$-valued matrix. 

\begin{proposition} 
\begin{itemize}
\item We have a canonical isomorphism $K_0(\mathcal{K}(\mathcal{E}))\simeq K_0(\mathcal{S}(\mathcal{E}))$.
\item The $G$-trace $\tr_G$ on $\mathcal{S}$ defines a group homomorphism $${\tr_G}_{\ast}: K_0(\mathcal{K}(\mathcal{E}))\rightarrow\mathbb{R}.$$
\end{itemize}
\end{proposition}

\begin{proof}
Proposition \ref{propositions of trace} (4) shows that $\mathcal{S}(\mathcal{E})$ is an ideal of $\mathcal{B}(\mathcal E).$ 
Since $\mathcal{S}(\mathcal{E})$ contains the rank one operators, then $\mathcal{K}(\mathcal{E})$ is the $C^{\ast}$-closure of $\mathcal{S}(\mathcal{E}).$ 
Let $J=\mathcal{K}(\mathcal{E}), \mathcal{J}=\mathcal{S}(\mathcal{E})$ and let $\tilde J, \tilde{J_0}$ be obtained by adjoining a unit. Note that $\tilde{J}=\mathcal{B}(\mathcal E).$ We claim that ${J_0}$ is stable under holomorphic functional calculus. To show the claim we essentially need to prove that if $a\in\tilde{J_0}$ is invertible in $\tilde J$, then $a^{-1}\in\tilde{J_0}$ \cite{Connes:1994zh}.  
Let $a^{-1}=\lambda I+r$, where $\lambda\in\C$, $I$ is the unit and $r\in J.$
Choose an $s\in J_0$ so that $\|a^{-1}-\lambda I-s\|<\min\{\frac{1}{\|a\|},1\}$. Then $\|1-\lambda a-as\|<1$ implies that $a(\lambda I+s)$ is invertible. So $\lambda I+s$ is also invertible and $s^{-1}\in\tilde J$, then $a^{-1}=(\lambda I+s)[a(\lambda I+s)]^{-1}$. Since $J_0$ is an ideal of $\tilde{J}$ we only need to show that $a(\lambda I+s)^{-1}\in\tilde{J_0}$. Let $x=a(\lambda I+s)\in J_0$, then $\|1-x\|<1$, then $\displaystyle x^{-1}=[1-(1-x)]^{-1}=\sum_{i=0}^{\infty}(1-x)^i\in\tilde{J_0}.$ The claim is proved.
Hence $\mathcal{S}(\mathcal{E})$ is a dense subalgebra of $\mathcal{K}(\mathcal{E})$ closed under holomorphic functional calculus, which implies that $K_{\ast}(\mathcal{K}(\mathcal{E}))=K_{\ast}(\mathcal{S}(\mathcal{E})).$

An element of $K_0(\mathcal{S}(\mathcal{E}))$ is represented by projection matrix with entries in $\mathcal{S}(\mathcal{E})$, on which there is a natural trace consisting of the composition of the matrix trace with $\tau$ on $\mathcal{S}(\mathcal{E})$.
Note that if the element was represented by the difference of two classes of matrices with entries in $\mathcal{S}(\mathcal{E})^+$, the algebra defined by adding a unit, then we define the trace of this extra unit to be $0.$
Hence we obtain a homomorphism ${\tr_G}_{\ast}: K_{\ast}(\mathcal{S}(\mathcal{E}))\rightarrow\R$ by the properties of the trace $\tau$.
\end{proof}

Composing with the $K$-theoretic index, $P$ has a numerical index given by the image of the map
$$K^0_G(C_0(X))\xrightarrow[]{{{\text{{K-theoretic index}}}}}
    {{K_0(\mathcal{S})}}
    \xrightarrow{{{{\tr_G}_{\ast}}}}\R$$
and this number depends only on the symbol class and the manifold according to Kasparov's $K$-theoretic index formula (Theorem \ref{k-homological}). We show that this number is in fact the $L^2$-index.

 \begin{proposition}\label{same ind}
 Let $P\in\Psi^0_{G,p}(X;E,E)$ be elliptic, then its $L^2$-index coincides with the trace of its $K$-theoretic index, i.e. $\ind P={\tr_G}_{\ast}(\mathrm{Ind}[P]).$
 \end{proposition}

\begin{proof}
Let $P=A$ and then $P=\bar A=\mathrm{Av}(cA)$ in \ref{KInd}. Then 
$$\mathrm{Ind} P=[\begin{pmatrix} P_0 P_0^{\ast} & P_0\sqrt{1-P_0^{\ast}P_0} \\ \sqrt{1-P_0^{\ast}P_0}P_0^{\ast} & 1-P_0^{\ast}P_0 \end{pmatrix}]-[\begin{pmatrix} 1 & 0\\ 0 &0 \end{pmatrix}].$$ 
We shall alter the matrix representatives without changing the equivalence class, so that we may apply $\tr_G$ to the $2\times2$-matrices.

Given $P_0\in\Psi^0_{G,p}(X;E_0,E_1)$ and using Proposition \ref{PDOprop}, there is a $Q\in\Psi^0_{G,p}$ so that $1-QP=S_0, 1-PQ=S_1.$ According to the boundary map construction in \cite{Connes:1990ht} section 2, we lift $\begin{pmatrix}0&-Q\\P&0\end{pmatrix}$ which is invertible in $M_2(\mathcal{B}(\mathcal{E})/\mathcal{S}(\mathcal{E}))$ to an invertible element $u=\begin{pmatrix}S_0 & -(1+S_0)Q \\ P & S_1\end{pmatrix}$ in $M_2(\mathcal{B}(\mathcal{E}))$ and then $$\mathrm{Ind} P\doteq[u\begin{pmatrix}1&0\\0&0\end{pmatrix}u^{-1}]-[\begin{pmatrix}0&0\\0&1\end{pmatrix}]=[\begin{pmatrix}S_0^2&S_0(1+S_0)Q\\P_0S_1&1-S_1^2\end{pmatrix}]-[\begin{pmatrix}0&0\\0&1\end{pmatrix}].$$
Therefore, ${\tr_G}_{\ast}(\mathrm{Ind} P)=\tr_G(S_0^2)+\tr_G(1-S_1^2)-\tau(1)=\tr_G(S_0^2)-\tr_G(S_1^2).$
Choose another $Q'\doteq2Q-QPQ$, then $1-Q'P_0=S_0^2, 1-PQ'=S_1^2$ with $S_0^2, S_1^2$ being smoothing operators. 
Then using Proposition \ref{prMS}, we conclude that $\tr_G(S_0^2)-\tr_G(S_1^2)=\ind P.$ Hence ${\tr_G}_{\ast}(\mathrm{Ind} P)=\ind P$. 
\end{proof}

\begin{remark}
Let $X=G/H$ be a homogeneous space of a unimodular Lie group $G$ (where $H$ is a compact subgroup).
In  \cite{Connes:1982} section 3, it was shown directly that the $L^2$-index depends only on the symbol class $[\sigma_P]$ of $P$ in $K^G_0(C_0(T^{\ast}X))$. 
Plus, there exists a homomorphism $i:K^G_0(C_0(T^{\ast}X))\rightarrow\R$ so that $i[\sigma_P]=\ind P.$
Note that the Poincar\'e duality between K-homology and K-theory gives rise to $K^G_0(C_0(T^{\ast}X))\simeq K_G^0(C_0(X)).$
So $L^2$-index essentially gives a homomorphism:
\begin{equation}\ind: K_G^0(C_0(X))\rightarrow\R.\end{equation}
\end{remark}

\begin{remark}
In this section we work on the cycles in $K_G^0(C_0(X))$ determined by odd self-adjoint elliptic pseudo-differential operators on $X$. 
If $Y$ is another proper cocompact $G$-manifold and if $E$ is a $G$-bundle where $L^2(Y,E)$ admits a $C_0(X)$-representation, so that $[(L^2(Y, E), Q)]\in K^0_G(C_0(X))$ with $Q\in\Psi^0_{G,p}(Y; E, E)$, we may carry out similar constructions to those in the section easily and there is no problem to define the $L^2$-index of $Q$. However, it is not clear how to define $L^2$-index for an arbitrary representative $(A,F)$ in a general representing cycle $[(A, F)]\in  K^0_G(C_0(X))$, where $A$ is a $C_0(X)$-algebra and $F$ is a general elliptic operator. Because we do not know the way to define pseudo-differential calculus for the $C^{\ast}$-algebra $A$ and we do not have Proposition \ref{PDOprop} for $F$, using which we calculate the $L^2$-index. But it should be possible to find a proper cocompact $G$-manifold $Y$ and pseudo-differential operator $Q$ on $Y$ so that $[(A, F)]=[(L^2(Y, E), Q)]\in K^0_G(C_0(X)).$ 
\end{remark}

\section{Reduction to the $L^2$-index of a Dirac type operator.}\label{5}

We shall show in this section that for any elliptic operator $P\in\Psi^0_{G,p}(X;E,E)$, there is a Dirac type operator $\tilde D$ satisfying $\ind P=\ind\tilde D.$
To do this, we show that $P$ and $\tilde D$ have the same $K$-theoretic index and then apply Proposition \ref{same ind}.

\begin{theorem}\cite{Kasparov:1983, Kasparov:2008}\label{k-homological}
{Let $X$ be a complete Riemannian manifold and let $G$ be a locally compact group acting on $X$ properly and isometrically. Let $P$ be a $G$-invariant elliptic operator on $X$ of order $0.$ Then 
\begin{equation}[P]=[\sigma_P]\otimes_{C_0(T^{\ast}X)}[D]\in K^{\ast}_G(C_0(X)),\label{k-homological formula}\end{equation} 
where $[D]$ is the equivalence class defined by the Dolbeault operator on $T^{\ast}X.$}
\end{theorem}

\begin{remark}
In (\ref{k-homological formula}), the ellipticity of $P=\begin{pmatrix}0 & P_0^{\ast} \\ P_0 & 0 \end{pmatrix}\in\Psi^0_{G,p}(X;E,E)$ (Definition \ref{defell}) implies that the symbol $\sigma_P=\begin{pmatrix}0 & \sigma_{P_0}\\ \sigma_{P_0} & 0 \end{pmatrix}$ defines an element of $KK^G(C_0(X), C_0(T^{\ast}X))$.
In fact, using the Hermitian structure on $E=E_0\oplus E_1$, we obtain $C_0(T^{\ast}X, \pi^{\ast}E)$, a Hilbert module over $C_0(T^{\ast}X)$, and the set of ``compact operators" is $C_0(T^{\ast}X, \Hom(\pi^{\ast}E, \pi^{\ast}E))$. Also $C_0(X)$ acts on $C_0(T^{\ast}X, \pi^{\ast}E_0\oplus\pi^{\ast}E_1)$ by pointwise multiplication.
Hence for all $f\in C_0(X)$, $(\sigma_P^2-I)f$ is compact by (\ref{ell}) and $[\sigma_P, f]=0$.
Therefore, the symbol $\sigma_P: \pi^{\ast}E\rightarrow\pi^{\ast}E$ defines the following element in $KK$-theory: 
$$[(C_0(T^{\ast}X, \pi^{\ast}E_0\oplus\pi^{\ast}E_1),\begin{pmatrix}0&\sigma_{P_0}^{\ast}\\ \sigma_{P_0}&0\end{pmatrix})]\in KK^G(C_0(X), C_0(T^{\ast}X)).$$ 

In (\ref{k-homological formula}), the {\em Dolbeault operator $D$} is a first order differential operator 
$D=\sqrt2(\bar\partial+\bar\partial^{\ast})$ acting on smooth sections of $\Lambda^{0,\ast}(T^{\ast}(T^{\ast}X))$, where  $\displaystyle\bar\partial=\frac{\partial}{\partial\bar z}=\frac12(\frac{\partial}{\partial\xi}+i\frac{\partial}{\partial x}).$
Denote by $H$ the Hilbert space of $L^2$-forms of bi-degree $(0,\ast)$ on $T^{\ast}X$ graded by the odd and even forms. 
Then $D$ is an order $1$ essentially self-adjoint operator on $H$. 
The $C^{\ast}$-algebra $C_0(T^{\ast}X)$ acts on $H$ by point-wise multiplication. 
\emph{The Dolbeault element} is the $K$-homological cycle given by
$$[(H,\frac{D}{\sqrt{1+D^2}})]\in K^0_G(C_0(T^{\ast}X))=KK^G(C_0(T^{\ast}X),\mathbb{C}).$$
\end{remark}

\begin{remark}\label{KAS}
Theorem \ref{k-homological} says that $[P]$ is given by the index pairing of the symbol with some fundamental (Dolbeult) operator on $T^{\ast}X$. This is the essence of the Atiyah-Singer index theorem.
When $X$ is compact with trivial group action, apply the map $$C^{\ast}: K^0(C(X))\rightarrow K^0(\C)$$ induced by the constant map $C: C(X)\rightarrow\C: f\mapsto f(\mathrm{pt})$ to both sides of (\ref{k-homological formula}). The left hand side of (\ref{k-homological formula}) is then the Fredholm index of $P$ and the right hand side is the intersection product of $[\sigma_P]\in K_0(C_0(T^{\ast}X))$ with $[D]\in K^0(C_0(T^{\ast}X))$. It is classical fact that $[\sigma_P]$ is viewed as some equivalence class of vector bundle $V$. Then the intersection product is the well-known Fredholm index of the Dirac operator $D$ with coefficients in $V$. 
\end{remark}  

The following $K$-theoretic index formula serves as an important corollary to Theorem \ref{k-homological}.

\begin{theorem}\cite{Kasparov:1983}\label{Kasparov}
{Let $X$ be a complete Riemannian manifold, on which a locally compact group $G$ acts properly and isometrically with compact quotient. Let $P$ be a properly supported $G$-invariant elliptic operator on $X$ of order $0.$ 
Then
\begin{align*}\mathrm{Ind} P&=[p]\otimes_{C^{\ast}(G,C_0(X))}j^G([P])\\
&=[p]\otimes_{C^{\ast}(G,C_0(X))}j^G([\sigma_P])\otimes_{C^{\ast}(G,C_0(T^{\ast}X))}j^G([D])\in K_{\ast}(C^{\ast}(G)).\end{align*} Where $p$ is the idempotent in $C^{\ast}(G,L^2(X,E))$ defined by $p=(c\cdot g(c))^{\frac{1}{2}}$ and $[D]$ is the Dolbeault element. }
\end{theorem}

Analogous to the vector bundle construction mentioned in Remark \ref{KAS} (See also \cite{ABP:1973} section 7), we define a $G$-bundle $V(\sigma_P)$ using the symbol $\sigma_P$ as follows. 
Let $B(X)\subset T^{\ast}X$ be the unit ball bundle with its boundary, that is, the sphere bundle $S(X)\subset T^{\ast}X$. 
A new manifold $\Sigma X$ is obtained by gluing two copies of $B(X)$ along their boundaries:
\begin{equation} \Sigma X=B(X)\cup_{S(X)} B(X). \end{equation} 
The action of $G$ on $T^{\ast}X$ extends naturally to $\Sigma X$ because $G$ acts on $X$ isometrically. 
%And a $G$-vector bundle over $\Sigma X$ is built out of $\sigma_P$ as follows.
The ellipticity of $P$ implies the invertibility of $\sigma_P|_{S(X)}$, the symbol restricted to $S(X)$. 
Define a $G$-vector bundle over $\Sigma X$ by the gluing map $\sigma_P$ on the boundary, that is, 
\begin{equation}\label{VSigmaP}
V(\sigma_P)=\pi^{\ast}E|_{B(X)}\cup_{\sigma_{P}|_{S(X)}}\pi^{\ast}E|_{B(X)}.
\end{equation}
Here $V(\sigma_P)$ defines an element in the representable $KK$-theory $RKK^0_G(X; C_0(X), C_0(\Sigma X))$. 
$V(\sigma_P)$ is $\Z/2\Z$-graded and is the direct sum of two bundles: 
$V(\sigma_{P_0})=\pi^{\ast}E_0|_{B(X)}\cup_{\sigma_{P_0}|_{S(X)}}\pi^{\ast}E_1|_{B(X)}$ 
and 
$V(\sigma_{P_0^{\ast}})=\pi^{\ast}E_1|_{B(X)}\cup_{\sigma_{P_0^{\ast}}|_{S(X)}}\pi^{\ast}E_0|_{B(X)}$.
There is a natural homomorphism 
$RKK^0_G(X; C_0(X), C_0(\Sigma X))\rightarrow KK^G(C_0(X), C_0(\Sigma X))$. Denote by $[V(\sigma_P)]$ as the equivalence class of $V(\sigma_P)$ either in $RKK^0_G(X; C_0(X), C_0(\Sigma X))$ and or in $KK^G(C_0(X), C_0(\Sigma X))$. 
We shall not distinguish the notations when it is clear from the context. 
In the proof Proposition \ref{KKsymbol} we shall see that as a $KK$-cycle, 
$$[V(\sigma_P)]=[(C_0(\Sigma X,V(\sigma_P)), 0)]\text{  in  }KK^G(C_0(X), C_0(\Sigma X)).$$

\begin{remark}
When $X$ is compact and when $G=\{e\}$, the inclusion $\C\rightarrow C(X)$ further reduces $\sigma_P$ to an element of $KK(\C, C_0(T^{\ast}X))$ by ``forgetting" the action of $C(X)$ on the Hilbert-$C(X)$ module $C_0(T^{\ast}X)$. 
Therefore,
$[\sigma_P]\in KK(\C, C_0(T^{\ast}X)))\simeq K_0(C_0(T^{\ast}X))$ maps to a vector bundle, trivial at infinity in $T^{\ast}X$. The bundle is constructed by gluing $\pi^{\ast}E|_{B(X)}$ and $\pi^{\ast}E|_{T^{\ast}X-B(X)^{\circ}}$ along the boundaries using the invertible map $\sigma_P|_{S(X)}$ and is the restriction of $V(\sigma_P)$ to $T^{\ast}X.$
\end{remark}

\begin{proposition} \label{KKsymbol}
The homomorphism  
\begin{align*}KK^G(C_0(X), C_0(T^{\ast}X))&\rightarrow KK^G(C_0(X), C_0(\Sigma X))\\
 [(C_0(T^{\ast}X, \pi^{\ast}E), \sigma_P)]&\mapsto [(C_0(\Sigma X,V(\sigma_P)), 0)]\end{align*} is induced by the inclusion map $i: C_0(T^{\ast}X)\rightarrow C_0(\Sigma X)$. 
\end{proposition}

\begin{proof}
First of all, the cycle $(C_0(\Sigma X,V(\sigma_P)), 0)$ defines an element of $KK^G(C_0(X), C_0(\Sigma X))$, 
because $f\cdot(0^2-\mathrm{Id}_{C_0(\Sigma X,V(\sigma_P))})$ is compact in the Hilbert-$C_0(\Sigma X)$-module $C_0(\Sigma X,V(\sigma_P)).$
Here, the compactness of the fiber of $\Sigma X$ over $X$ is important. The argument fails when replacing $\Sigma X$ by $T^{\ast}X$. For example, $(C_0(\Sigma X,V(\sigma_P)|_{T^{\ast}X}), 0)$ does not define an element in $KK^G(C_0(X), C_0(T^{\ast}X))$.

Without loss of generality, we may assume $\sigma_P$ satisfies that 
$$\sigma_P^2=1\text{ on }S(X)\text{ and }\|\sigma_P\|\le1.$$ 
Using the standard boundary map construction in the exact sequence of $K$-theory, we obtain the following projection $Q$ using the unitary 
$u=\begin{pmatrix}\sigma_P & -\sqrt{1-\sigma_P^2} \\ \sqrt{1-\sigma_P^2} & \sigma_P \end{pmatrix}\in M_2(C_0(T^{\ast}X, \pi^{\ast}E))$:  
$$Q\doteq u\begin{pmatrix}1 & 0 \\ 0 & 0\end{pmatrix}u^{-1}=\begin{pmatrix}\sigma_{P}^2 & \sigma_{P}\sqrt{1-\sigma_P^2} \\ \sqrt{1-\sigma_P^2}\sigma_P & 1-\sigma_P^2\end{pmatrix}.$$

Recall that the bundle $(V(\sigma_P)|_{T^{\ast}X}$ glued by $\sigma_P$ is given by the image of the projection $Q$ on $\pi^{\ast}E\oplus\pi^{\ast}E$, that is, 
$$C_0(T^{\ast}X,V(\sigma_P)|_{T^{\ast}X})=QC_0(T^{\ast}X, \pi^{\ast}E\oplus\pi^{\ast}E).$$
Let $w=u\begin{pmatrix} \sigma_P & 0 \\0 & 1\end{pmatrix}u^{\ast}.$ Then $QwQ=w, (1-Q)w(1-Q)=1-Q,$ and $Qw(1-Q)=(1-Q)wQ=0.$
Then 
\begin{align*}
  &[(C_0(T^{\ast}X, \pi^{\ast}E), \sigma_P)]=[(C_0(T^{\ast}X, \pi^{\ast}E), \sigma_P)]+[(C_0(T^{\ast}X, \pi^{\ast}E), 1)]\\
=&[(C_0(T^{\ast}X, \pi^{\ast}E\oplus\pi^{\ast}E), \begin{pmatrix}\sigma_P & 0 \\ 0 & 1\end{pmatrix})] =[(C_0(T^{\ast}X, \pi^{\ast}E\oplus\pi^{\ast}E), w)]  \\
=&[(QC_0(T^{\ast}X, \pi^{\ast}E\oplus\pi^{\ast}E), x)]+ [((1-Q)C_0(T^{\ast}X, \pi^{\ast}E\oplus\pi^{\ast}E), 1-Q)] \\
=&[(QC_0(T^{\ast}X, \pi^{\ast}E\oplus\pi^{\ast}E), x)]\\  
\rightarrow&[(C_0(\Sigma X,V(\sigma_P)), \tilde x)]=[(C_0(\Sigma X,V(\sigma_P)), 0)].
\end{align*}
Here, the arrow in the last line comes from the following fact. The Hilbert $C_0(T^{\ast}X)$-module $C_0(T^{\ast}X, V(\sigma_P)|_{T^{\ast}X})$ maps to the Hilbert $C_0(\Sigma X)$-module $C_0(\Sigma X, V(\sigma_P))$ under the map 
$i_{\ast}: KK^G(C_0(X), C_0(T^{\ast}X))\rightarrow KK^G(C_0(X), C_0(\Sigma X))$
induced from the inclusion $i$. 
The last equality follows from the operator homotopy $t\rightarrow t\tilde x$ and the observation that $(C_0(\Sigma X,V(\sigma_P)), t\tilde x)$ is a Kasparov $(C_0(X), C_0(\Sigma X))$ module for all $t\in[0,1]$.
The proof is complete.
\end{proof}

The Dolbeault operator on $T^{\ast}X$ extends to the proper cocompact $G$-manifold $\Sigma X$, which also has an almost complex structure.
We just glue two Dolbeault operators on $B(X)\subset T^{\ast}X$ along the boundary 
(the normal directions of $S(X)$ in $B(X)$ need to switch signs on different pieces). 
The new Dolbeault operator $\bar D$ is clearly $G$-invariant and defines an element 
\begin{equation}\label{barD}
[\bar D]=[(L^2(\Sigma X, \Lambda^{0,\ast}(T^{\ast}(\Sigma X))), \frac{\bar D}{\sqrt{1+\bar D^2}})]
\end{equation}
in $KK^G(C_0(\Sigma X), \C)$.
(In section 6 we shall not distinguish $[\bar D]$ and $[D].$) 
The following proposition is obvious. 

\begin{proposition}
The inclusion $i: C_0(T^{\ast}X)\rightarrow C_0(\Sigma X)$ induces the natural map 
\begin{equation}i^{\ast}: KK^G(C_0(\Sigma X), \C)\rightarrow KK^G(C_0(T^{\ast}X), \C):[\bar D]\mapsto[D].\end{equation}
\end{proposition}

\begin{corollary}\label{imcor}
Assuming the same notations and conditions in the $K$-homological formula in Theorem \ref{k-homological}, we have
\begin{enumerate}
\item The elliptic pseudo-differential operator $P$ is in the same K-homology class as the intersection product $[V(\sigma_P)]\otimes [\bar D]$ in the image of the map
$$KK^G(C_0(X), C_0(\Sigma X))\times KK^G(C_0(\Sigma X), \C)\rightarrow KK^G(C_0(X), \C).$$ 
\item The operator $P$ relates to a Dirac type operator $\bar D_{V(\sigma_P)}$, that is, the Dolbeault operator $\bar D$ on $\Sigma X$ twisted by the bundle $V(\sigma_P)$ over $\Sigma X$, in the following sense:
 \begin{equation}[P]=j^{\ast}[\bar D_{V(\sigma_P)}]\label{pair}\end{equation} 
where $j^{\ast}: KK^G(C_0(\Sigma X), \C)\rightarrow KK^G(C_0(X), \C)$ is induced by the inclusion $j: C_0(X)\rightarrow C_0(\Sigma X).$
\end{enumerate}
\end{corollary}
  
\begin{proof}
The first statement is a result of Theorem \ref{k-homological} as well as the functorality of intersection products 
$$[P]=[\sigma_P]\otimes_{C_0(TX)}[D]=[\sigma_P]\otimes_{C_0(TX)}i^{\ast}[\bar D]=i_{\ast}[\sigma_P]\otimes_{C_0(\Sigma X)}[\bar D]=[V(\sigma_P)]\otimes_{C_0(\Sigma X)}[\bar D].$$  
To prove the second statement, we calculate 
\begin{equation}\label{productinproof}
[V(\sigma_P)]\otimes_{C_0(\Sigma X)}[\bar D]=[(C_0(\Sigma X, V(\sigma_P)),\phi_1, 0)]\otimes_{C_0(\Sigma X)}[(L^2(\Sigma X, \Lambda^{0,\ast}(T^{\ast}(\Sigma X))), \phi_2, F)],
\end{equation}
where $F\doteq\frac{\bar D}{\sqrt{1+\bar D^2}}.$
We denote by $[(H, \eta, I)]$ the $KK$-product appeared in (\ref{productinproof}).   

According to the definition of $KK$-product,  
$$H=C_0(\Sigma X, V(\sigma_P))\otimes_{C_0(\Sigma X)}L^2(\Sigma X, \Lambda^{0,\ast}(T^{\ast}(\Sigma X)))$$ 
and the operator $I$ needs to satisfy the following two conditions \cite{Skandalis:1984}:
\begin{enumerate} 
\item $I$ is an $F$-connection;
\item $I$ has the property $\eta(a)[0\otimes1, I]\eta(a)\ge0$ modulo $\mathcal{K}(H)$.
\end{enumerate}
By Kasparov's stabilization theorem, there is a $C_0(\Sigma X)$-valued projection $Q$ such that $C_0(\Sigma X, V(\sigma_P))=Q(\oplus_{1}^{\infty} C_0(\Sigma X))$. 
Therefore, 
\begin{align*}
H=&Q(\oplus_{1}^{\infty} C_0(\Sigma X))\otimes_{C_0(\Sigma X)}L^2(\Sigma X, \Lambda^{0,\ast}(T^{\ast}(\Sigma X)))\\
=&\phi_2(Q)(\oplus_{1}^{\infty} L^2(\Sigma X, \Lambda^{0,\ast}(T^{\ast}(\Sigma X))),
\end{align*}
where, $\phi_2(Q)$, by definition, acts by matrix multiplication and point-wise multiplication.

We claim that 
\begin{equation}
I=\phi_2(Q)(\oplus_{1}^{\infty} F)\phi_2(Q)\label{claim}
\end{equation}
The statement is proved if (\ref{claim}) is true. 
In fact, one needs only to observe that 
\begin{align*}
H=&\phi_2(Q)(\oplus_{1}^{\infty} L^2(\Sigma X, \Lambda^{0,\ast}(T^{\ast}(\Sigma X)))\\
=&L^2(\Sigma X, \Lambda^{0,\ast}(T^{\ast}(\Sigma X))\otimes V(\sigma_P))
\end{align*} 
and 
$$\phi_2(Q)(\oplus_{1}^{\infty}\bar D)\phi_2(Q)=\bar D_{V(\sigma_P)}\text{ on }H.$$ 
To prove the claim (\ref{claim}), it is sufficient to show the following observations.
\begin{itemize}
\item $(I^2-1)\eta(f)\in\mathcal{K}(H),$ for all $f\in C_0(X);$
\item $[I, \eta(f)]\in\mathcal{K}(H),$ for all $f\in C_0(X);$ 
\item $[\tilde{T_{\xi}}, F\oplus I]\in\mathcal{K}(L^2(\Sigma X, \Lambda^{\ast}(\Sigma X))\oplus H),\forall \xi\in C_0(\Sigma X, V(\sigma_P)),$ where 
$$\tilde{T_{\xi}}=\begin{pmatrix}0&T^{\ast}_{\xi}\\T_{\xi}&0\end{pmatrix}\in\mathcal{B}(L^2(\Sigma X,  \Lambda^{0,\ast}(T^{\ast}(\Sigma X))\oplus H), T_{\xi}\in\mathcal{B}(L^2(\Sigma X,  \Lambda^{0,\ast}(T^{\ast}(\Sigma X)), H)$$ 
is defined by $T_{\xi}(\eta)=\xi\hat{\otimes}\eta\in H.$
\end{itemize}
\end{proof}  

\begin{proposition}\label{ADE}
Let $P$ be a properly supported $G$-invariant elliptic pseudo-differential operator of order $0$, $\bar D$ be the Dolbeault operator on $\Sigma X$ defined in (\ref{barD}) and $V(\sigma_P)$ be the $G$-vector bundle over $\Sigma X$ defined in (\ref{VSigmaP}) Then $P$ and $D_{V(\sigma_P)}$ have the same $L^2$-index, that is,
\begin{equation}\ind P=\ind D_{V(\sigma_P)}.\end{equation}
\end{proposition}

\begin{proof}
In view of Corollary \ref{imcor}, the cycle 
$$[(L^2(\Sigma X,  \Lambda^{0,\ast}(T^{\ast}(\Sigma X))\otimes V(\sigma_P)), \bar D_{V(\sigma_P)})]$$ 
represents as the intersection product $[V(\sigma_P)]\otimes [\bar D]$, which is the same as $[(L^2(X,E),P)]$ in $K_G^0(C_0(X)).$
This implies that $\mathrm{Ind} P=\mathrm{Ind} D_{V(\sigma_P)}$ and the statement is proved by taking the trace of the $K$-theoretic indices.
\end{proof}

\section{Local index formula.}\label{6}

\subsection{$L^2$-index of Dirac type operators.}

Using Proposition \ref{ADE}, to find a cohomological formula for the $L^2$-index of $P$, it is sufficient to figure out a formula for Dirac type operators.  Let $M$ be an even-dimensional ($\dim M=n$) proper cocompact $G$-manifold with a $G$-Clifford bundle $V$, which is a $\mathrm{\C l}(T^{\ast}M)$-module via Clifford multiplication. Here $\mathrm{\C l}(T^{\ast}M)=\Cl(T^{\ast}M)\otimes\C$ is the complex Clifford algebra generated by $T^{\ast}M.$  We construct $\D$, a Dirac type operator acting on sections in $V$. 
Let $\nabla$ be the  $G$-invariant Levi-Civita connection on $TM$ , which  can be extended to $\Cl(T^{\ast}M)$.
Let $\nabla^V$ be the $G$-invariant {\em Clifford connection} on $V$, i.e. $[\nabla^V, \cl(a)]=\cl(\nabla a), a\in C_c^{\infty}(M,\Cl(T^{\ast}M)).$
A {\em Dirac operator} 
$\D: C_c^{\infty}(M, V)\rightarrow C_c^{\infty}(M, V)$
is defined as the composition of the connection $\nabla^V$ and the Clifford multiplication $\cl: C_c^{\infty}(M,T^{\ast}M\times V)\rightarrow C_c^{\infty}(M,V)$ 
by $$\D=\sum_i\cl(e^i)\nabla^V_{e_i},$$ where $\{e_i\}$ forms an orthonormal basis of the bundle $TM$ and $\{e^i\}$ is the dual basis of $T^{\ast}M$.
Here, $V=V_0\oplus V_1$ is $\Z/2\Z$ graded and $\D$ is essentially self-adjoint with an odd grading, in particular, $\D=\begin{pmatrix}0 &\D_0^{\ast} \\ \D_0 & 0\end{pmatrix}: L^2(M,V)\rightarrow L^2(M,V)$.
The $L^2$-index of $\D$ is expressed by the McKean-Singer formula (\ref{MS}) which is independent of $t$:  
\begin{equation}\ind \D=\str_G(e^{-t\D^2}),\label{mckean singer}\end{equation}
where $\str_G(\begin{pmatrix}a&b\\c&d\end{pmatrix})=\tr_G(a)-\tr_G(d)$ and $\D^2=\begin{pmatrix}\D_0^{\ast}\D_0 & 0 \\ 0 & \D_0\D_0^{\ast}\end{pmatrix}.$

Let $R^V=(\nabla^V)^2\in\Lambda^2(M, \Hom V)$ be the curvature tensor of the Clifford connection $\nabla^V$, then
\begin{align*}
\D^2&=-\sum_{i}(\nabla_{e_i}^V)^2+\sum_i\nabla_{\nabla_{e_i}e_i}^V+\sum_{i<j}\cl(e^i)\cl(e^j)R^V(e_i,e_j)\doteq\Delta^V+\sum_{i<j}\cl(e^i)\cl(e^j)R^V(e_i,e_j)
\end{align*}
is a generalized Laplacian.  
Let $S$ be the spinor (irreducible) representation of $\Cl(T_x^{\ast}M)$. It is a standard fact that 
$\Hom S=S\otimes S^{\ast}=\mathrm{\C l}(T_x^{\ast}M).$
The fiber of the Clifford module $V$ at $x$ has the decomposition 
$V_x=S\otimes W.$
Here $W$ is the set of vectors in $V_x$ that commute with the action of $\mathrm{\C l}(T_x^{\ast}M).$
Therefore on the endomorphism level we have 
\begin{equation}\Hom V_x=\mathrm{\C l}(T_x^{\ast}M)\otimes\Hom W.\label{decomposition}\end{equation} 
Here $\Hom_{\mathrm{\C l}(T_x^{\ast}M)}(V_x)\doteq\Hom W$ is made of the transformations of $V_x$ that commute with $\mathrm{\C l}(T_x^{\ast}M)$. 
According to \cite{BGV} Proposition 3.43, the curvature $R^V$ decomposes under the isomorphism (\ref{decomposition}) into
\begin{equation}R^V=R^S+F^{V/S} \label{decomcur}\end{equation}
where $R^S(e_i, e_j)=\frac14\sum_{kl}(R(e_i,e_j)e_k, e_l)\cl^k\cl^l$ is the action of the Riemannian curvature $R\doteq\nabla^2$ of $M$ on the bundle $V$ and $F^{V/S}\in\Lambda^2(M, \Hom_{\mathrm{\C l}}V)$ is the twisting curvature of the Clifford connection $\nabla^V.$ 
According to the Lichnerowicz Formula, \cite{BGV} Proposition 3.52, the generalized Laplacian is calculated by:
\begin{equation}\label{LF}\D^2=-\sum_{i=1}^n(\nabla^V_{e_i})^2+\sum_i\nabla_{\nabla_{e_i}e_i}^V+\frac14 r_M+\sum_{i<j}F^{V/S}(e_i, e_j)\cl(e_i)\cl(e_j),\end{equation}
where $F^{V/S}(e_i, e_j)\in\Hom_{\mathrm{\C l}}V$ are the coefficients of the twisting curvature $F^{V/S}.$

Let the \emph{heat kernel} $k_t$ be the Schwartz kernel of the solution operator $e^{-t\D^2}$ of the heat equation 
$\displaystyle\frac{\partial}{\partial t}u(t,x)+\D^2u(t,x)=0$. It is a smooth map $M\times M\rightarrow \Hom(V,V)$ satisfying 
$\displaystyle e^{-t\D^2}f(x)=\int_Mk_t(x,y)f(y)\dif y.$ 
Hence $$\ind\D=\int_M c(x)\str k_t(x,x)\dif x.$$
We have the following properties of the heat kernel.
 
 \begin{lemma}\label{ll}
\begin{enumerate}
\item For $f(x)\in L^2(M)$, $e^{-t\D^2}f$ is a smooth section;
\item The kernel $k_t(x,y)$ of $e^{-t\D^2}$ tends to the $\delta$ function weakly, i.e. \par 
$\displaystyle e^{-t\D^2}s(x)=\int_Mk_t(x,x_0)s(x_0)\dif x_0\to s(x)$ uniformly on a compact set in $M$ as $t\to0.$
\end{enumerate}
\end{lemma}

\begin{proof}
We have proved that the Schwartz kernel of $ce^{-t\D^2}$ is smooth in Lemma \ref{exp tr}. 
So 
$$(e^{-t\D^2}f)(x)=\int_{G\times M}c(g^{-1}x)k_t(x,y)f(y)\dif y\dif g\doteq\int_Gh_g(x)\dif g,$$ 
where $h_g(x)=\int_Mc(g^{-1}x)k_t(x,y)f(y)\dif y$ is smooth in $x\in M$ for fixed $g\in G$. 
Using the fact that $e^{-t\D^2}$ is a bounded operator and that $c(x)$ is smooth and compactly supported, we conclude that  $h_g(x)$ depends smoothly on $g\in G$ . 
Let $K$ be any compact neighborhood of $x$, then by the properness of the group action, the set 
$$Z\doteq\{g\in G| c(g^{-1}x)\neq0, x\in K, g\in G\}$$ 
is compact and then $\displaystyle(e^{-t\D^2}f)(x)=\int_Z h_g(x)\dif g$ is smooth for $x\in K$. 
Therefore the first statement is proved.

To prove the second one, let $u$ be a smooth function with norm $1.$ 
Then $\displaystyle<e^{-t\D^2}u, u>=\int_{\lambda\in\mathrm{sp}(\D)}e^{-t\lambda^2}\dif P_{u,u}$, where $\mathrm{sp}(\D)$ means the spectrum of $\D$. 
Since the set of integrals for $0<t\le1$ is bounded by $1$, then by the dominated convergence theorem, $$<e^{-t\D^2}u, u>\to \int_{\lambda\in\mathrm{sp}(\D)}1\dif P_{u,u}=<u,u>\text{ as }t\to0.$$
\end{proof}
 
The heat kernel on $\R^n$ of $\displaystyle u_t-\sum_{i=1}^n\frac{\partial^2}{\partial^2 x_i}=0$, which is
\begin{equation}p_t(x,y)=\frac{1}{(4\pi t)^{n/2}}e^{-d(x,y)^2/4t},\label{1htkn}\end{equation}
suggests a first approximation for the heat kernel on $M.$
The small time behavior of the heat kernel $k_t(x,y)$ for $x$ near $y$ depends on the local geometry of $x$ near $y$. This is made precise by the \emph{asymptotic expansion} for $k_t(x,y).$

\begin{definition}[\cite{Roe:1998ad}]
Let $B$ be a Banach space with norm $\|\cdot\|$ and $f:\R^+\rightarrow B: t\mapsto f(t)$ be a function. 
A formal series $\displaystyle\sum_{k=0}^{\infty}a_k(t)$ with $a_k(t)\in E$ is called \emph{an asymptotic expansion} for $f$, denoted by $\displaystyle f(t)\sim\sum_{i=0}^{\infty}a_k(t)$, if for any $m>0$, there are $M_m$ and $\epsilon_m>0$. So that for all $l\ge M_m, t\in(0,\epsilon_m]$, we have 
$$\|f(t)-\sum_{k=0}^l a_k(t)\|\le Ct^m.$$
\end{definition}

When $M$ is compact and when $B=C^0(M, \Hom(V,V))$ has $C^0$-norm $\|f\|=\sup_{x\in M}|f(x)|$, it is the standard fact that the heat kernel $k_t(x,x)$ of $e^{-t\D^2}$ has an asymptotic expansion 
$$k_t(x,x)\sim \frac{1}{(4\pi t)^{n/2}}\sum_{j=0}^{\infty}t^j a_j(x)$$ where $a_j(x)\in\Hom(V_x,X_x), x\in M$ are smooth sections (\cite{Roe:1998ad} Theorem 7.15).
In our case, this theorem is formulated as follows.

\begin{theorem}\label{later1}
Let $M$ be a proper cocompact Riemannian $G$-manifold and $\D$ be an equivariant Dirac type operator acting on the sections of a Clifford bundle $V$, and $k_t$ be the heat kernel of $\D$. There is an asymptotic expansion for $c(x)k_t(x,x)$ under the $C^0$-norm $\|f\|=\sup_{x\in M}|f(x)|: $
\begin{equation}c(x)k_t(x,x)\sim c(x)\frac{1}{(4\pi t)^{n/2}}\sum_{j=0}^{\infty}t^j a_j(x)\label{AE}\end{equation}
 where $a_j\in C^{\infty}(M,\Hom V)$ and $a_j(x)$ depends only on the the geometry at $x$ (involving metrics, connection coefficients and their derivatives). 
In particular $a_0(x)=1.$
The asymptotic expansion works for any $C^l$-norm for $l\ge0$. (We only need and prove the case when $l=0$.)
\end{theorem}

To prove Theorem \ref{later1} we constructe an ``approximating heat kernel". 
The proof is a modification of the case of operators on compact manifold (\cite{Roe:1998ad} Theorem 7.15 or \cite{BGV} Chapter 2).
Now $k_t(x,y)$ satisfies the heat equation:
\begin{equation}\frac{\partial}{\partial t}k_t(x,y)+\D^2k_t(x,y)=0, k_0(x,y)=\delta_{y}(x)\label{heat eq}\end{equation}  
where $\D$ operates on the $x$-coordinate only. We fix $y$ and denote it by $x_0$ and solve this equation locally on a coordinate neighborhood $O_{x_0}$ of $x_0$ with $x\in O_{x_0}$. 
We approximate the heat kernel $k_t(x,x_0), x\in O_{x_0}$ locally by looking for a formal solution 
\begin{equation}p_t(x,x_0)\sum_{i=0}^{\infty}t^ib_i(x)\label{asy}\end{equation} 
to the equation (\ref{heat eq}), where 
$\displaystyle p_t(x,x_0)=\frac{1}{(4\pi t)^{\frac{n}{2}}}e^{-\frac{r^2}{4t}}$ with $r=|{\bf x}|=d(x,x_0)$ 
is the heat kernel on Euclidean space (\ref{1htkn}). 
Denote by $\displaystyle s_t(x,x_0)=\sum_{i=0}^{\infty}t^ib_i({x})$ in (\ref{asy}) and so the heat kernel is written as \begin{equation}k_t(x,x_0)=p_t(x,x_0)s_t(x,x_0).\label{k=ps}\end{equation}

According to \cite{Roe:1998ad} equation 7.16, $\D^2$ in (\ref{LF}) on $O_{x_0}$ is calculated by \begin{equation}[\frac{\partial}{\partial t}+\D^2](p_ts_t)=p_t[\frac{\partial}{\partial t}+\D^2+\frac{r}{4gt}\frac{\partial g}{\partial r}+\frac{1}{t}\nabla_{r\frac{\partial}{\partial r}}]s_t.\label{mid term}\end{equation}
when operating on (\ref{k=ps}), where $r=|{\bf x}|$, $g=\det(g_{ij})$ and $(g_{ij})$ is the Riemannian metric on $M$. 
To find the formal solution (\ref{asy}), set the right hand side of (\ref{mid term}) to be $0$. Then the comparison of the coefficients of terms containing $t^i$ for each $i\ge0$ enables us to find $b_i$ inductively via \cite{Roe:1998ad} equation (7.17): 
\begin{equation} \label{hoho}
\nabla_{\frac{\partial}{\partial r}}(r^ig^{\frac14}b_i({x}))=\begin{cases}0 & i=0\\ -r^{i-1}g^{\frac14}\D^2b_{i-1}({x}) & i>0\end{cases}
\end{equation}

(1) (Solve $\alpha_0(x)$) It is trivial to see that $p_t(x,x_0)=\frac{1}{(4\pi t)^{\frac{n}{2}}}e^{-\frac{r^2}{4t}}\to\delta_{x_0}(x)$ uniformly as $t\to0+$. From Lemma \ref{ll}, $k_t(x,x_0)\to \delta_{x_0}(x)$ uniformly as $t\to0+$ for all $x\in K$, where $K\subset X$ is any compact subset. 
Therefore $b_0({x_0})=1$ necessarily. The first line in (\ref{hoho}) indicates that $g^{\frac14}b_0({x})=g({x_0})^{\frac14}b_0({x_0})=1$, and then $b_0({x})=g^{-\frac14}(x)$ is determined by $b_0({x_0}).$

(2) (Solve $b_i(x), i>0$) Inductively the smoothness of $b_i$ implies the uniqueness of the smooth solution $b_{i+1}.$ In fact, when solving the equation in (\ref{hoho}), the constant term has to be $0$ otherwise $b_{i+1}$ is not smooth at $r=0$. Then $b_{i+1}$ is smooth except that it may blow up at $0.$ But by setting $r=0$ in the second line in (\ref{hoho}) we have $b_{i+1}({x_0})=-\frac1j(\D^2b_{i})({x_0})$ which makes sense if $b_i$ is smooth. Therefore, there exists a sequence of smooth sections $\{b_i({x})\}$ in $\Hom(V_{x_0}, V_x)$ uniquely determined by $b_0({x_0})=1.$

Note that $b_i$s are defined on a coordinate neighborhood $O_{x_0}$ and depend smoothly on the local geometry around $x_0$. 
For example, $b_1(x)=\frac{1}{6}k(x)-K(x)$, 
where $k$ is scalar curvature and $K$ satisfies $\D^2=\Delta+K$. 

Denote $b_i(x)$ by $b_i(x, x_0), x\in O_{x_0}$. Now for any $x_0\doteq y\in M$, we obtain a formal solution $b_i(x,y)$ which smoothly depends on both $x$ and $y$ for $x\in O_y$.
Choose $O'\subset M\times M$ such that $\{(x,x)|x\in M\}\subset O'\subset \cup_{y\in M}O_{y}$ and choose 
$$\phi(x,y)\in C^{\infty}(M\times M)\text{ so that }\phi(x,y)=\begin{cases}1 & (x,y)\in O' \\0 & (x,y)\notin \cup_{y\in M}O_y\end{cases}.$$ 
This definition is based on a cutoff function used to define the approximate heat kernel in \cite{BGV} Definition 2.28.

\begin{definition} \label{approxhk}
Let (\ref{k=ps}) be the true heat kernel. The approximating heat kernel is
\begin{equation}h^n_t(x,y)=p_t(x,y)\sum_{i=0}^{n}t^i a_i(x,y),\label{apphk}\end{equation}
where $a_i(x,y)=\phi(x,y)b_i(x,y)\in C^{\infty}(M\times M)$ and supported in a neighborhood of the diagonal.
\end{definition}

With the previous set up we may state the following lemma, which implies Theorem \ref{later1} when setting $x=y.$

\begin{lemma} \label{impro}
Let $k_t(x,y)$ be the heat kernel and $h_t^n(x,y)$ be the one in (\ref{apphk}). Let $c\in C_c^{\infty}(M)$ be a cutoff function of the proper cocompact $G$-manifold $M$. Choose $\bar{c}\in C^{\infty}_c(M)$ satisfying $c(x)\bar{c}(x)=c(x), x\in M.$ For all $m>0,$ there is a $N_m$, so that for all $l>N_m$ and $t\in(0,1],$ 
\begin{equation}\|c(x)h_t^l(x,y)\bar{c}(y)-c(x)k_t(x,y)\bar{c}(y)\|<Ct^m\label{Oh}\end{equation}
where $\|f\|=\sup_{x,y\in M}|f(x,y)|.$ 
\end{lemma}

\begin{proof}
For all $m$, let $N_m>\mathrm{max}\{n+1,m+\frac{n}{2}\}$, where $n=\dim M$. 
By definition $h^{N_m}_t(x,y)$ approximately satisfies the heat equation in the sense that 
\begin{equation}
(\frac{\partial}{\partial t}+\D^2)h^{N_m}=
t^{N_m} p_t(x,y)\D^2 a_{N_m}(x,y)+O(t^{\infty})
\doteq r_t(x,y),\label{partial sol}
\end{equation} 
where the first term in (\ref{partial sol}) comes from the calculation of the formal solution.
In fact, using (\ref{mid term}), (\ref{hoho}), we have $\displaystyle(\frac{\partial}{\partial t}+\D^2)[p_t(x,y)\sum_{j=0}^{N_m}t^jb_j(x,y)]=t^{N_m} p_t(x,y)\D^2 b_{N_m}(x,y)$.
What remains $O(t^{\infty})$ is of order $t^{\infty}$, because this term contains the derivatives of $\phi$, which are of $0$-value for $x$ near $y$, and $p_t(x, y), (x\ne y)$, which decreases faster than any positive power $t^k$ as $t\to0+$.
$r_t(x,y)$ has the following properties:

(1) The remainder $r_t(x,y)$ is smooth for any fixed $t>0$. This is because $p_t(x,y)$ in (\ref{1htkn}) and $a_i(x,y)$s in Definition \ref{approxhk} are smooth functions, for all $t>0$.
 
(2) Denote the $k$th Sobolev norm on $C^m(M\times M)$ by $\|\cdot\|_k.$ Then for all fixed $t>0$ and for all $k$: 
$\|c(x)r_t(x,y)\bar{c}(y)\|_k$ exists. This is because $c(x)r_t(x,y)\bar{c}(y)$ is smooth and compactly supported on $M\times M.$

(3) \label{OMG}We have the estimate $$\|c(x)r_t(x,y)\bar{c}(y)\|_{\frac{n}{2}+1}<Ct^m$$ uniformly for all $t\in(0,1].$
In fact, in the first term $c(x)t^{N_m} p_t(x,y)(\D^2 a_{N_m}(x,y))\bar{c}(y)$ of $c(x)r_t(x,y)\bar{c}(y)$, only $t^{N_m} p_t(x,y)$ depends on $t$, it is sufficient to know the order of $t$ in the $k$th derivative (in $x$ or $y$) of $t^{N_m} p_t(x,y)$, where $k\le\frac{n}{2}+1$ and the order is: $t^{N_m}t^{-\frac{n}{2}}t^{-k}=t^{N_m-\frac{n}{2}-k}$. 
      So $$\|c(x)t^{N_m} p_t(x,y)(\D^2 a_{N_m}(x,y))\bar{c}(y)\|_{\frac{n}{2}+1}\le\sum_{k=0}^{\frac{n}{2}+1}c_kt^{N_m-\frac{n}{2}-k}.$$ 
      Since $N_m>n+1$, there are no terms of non-positive order in $t$ on the right hand side. In addition, since $N_m>\frac{n}{2}+m$, then for all $t\in(0,1]$ , there is a constant $C_1$ so that
$$\|c(x)t^{N_m} p_t(x,y)(\D^2 a_{N_m}(x,y))\bar{c}(y)\|_{\frac{n}{2}+1}\le C_1t^{N_m-\frac{n}{2}}\le C_1t^{m}.$$
      The derivatives of $c(x)O(t^{\infty})\bar{c}(y)$ do not have any terms containing negative power of $t$ so $\|c(x)O(t^{\infty})\bar{c}(y)\|_{\frac{n}{2}+1}<C_2t^m$ for all $t\in(0,1].$ So property (3) is proved.

Next, we use $r_t(x,y)$ to relate $k_t(x,y)$ and $h^{N_m}_t(x,y)$ in the following claim:

{\bf Claim:} There is a unique smooth solution for the following equation: 
\begin{equation}\begin{cases}(\frac{\partial}{\partial t}+\D^2)u_t(x,y)=r_t(x, y) \\ u_0(x,y)=0\end{cases}\label{!}\end{equation} 
Here, $u_t(x,y)$ is regard as a function of $t$ and $x.$

In fact,
It is trivial to check that $u_1=\int_0^te^{-(t-\tau)\D^2}r_{\tau}(x, x_0)\dif\tau$ is smooth and satisfies the equation. 
If $u_2$ is another smooth solution, then $u=u_1-u_2$ satisfies $(\frac{\partial}{\partial t}+\D^2)u=0, u_0=u(t=0)=0.$
Hence 
$$\frac{\dif}{\dif t}\|u\|^2_{L^2}=\frac{\dif}{\dif t}<u,u>=-<u,\D^2u>-<\D^2u,u>=-2\|\D u\|^2_{L^2}$$ 
implies that $\|u\|^2$ is non-decreasing in $t$, and so $\|u(t=0)\|=0$ forces $u=u_1-u_2=0.$ So the claim is proved.

Since $h^{N_m}_t(x,y)-k_t(x,y)$ is also a solution to the equation (\ref{!}),
by the uniqueness of solution we have that
$\displaystyle h^{N_M}_t(x,y)-k_t(x,y)=\int_0^te^{-(t-\tau)\D^2}r_{\tau}(x, y)\dif\tau.$
Then for all $t\in(0,1]$,
$$ \|c(x)k_t(x,y)\bar{c}(y)-c(x)h^{N_m}_t(x,y)\bar{c}(y)\|_{\frac{n}{2}+1}\le t\sup\{\|c(x)r_{\tau}(x,y)\bar{c}(y)\|_{\frac{n}{2}+1}|0\le\tau\le t\}\le Ct^m,$$
where the second inequality is because of property (3).

By the Sobolev embedding theorem, for all $p>\frac{n}{2}$, $\|u\|\le C_0\|u\|_p$ for $u\in H^p$, where $\|\cdot\|$ is the $C^0$ sup norm and $\|\cdot\|_p$ is the Sobolev $p$-norm. Therefore,
\begin{align*}
\|c(x)k_t(x,y)\bar{c}(y)-c(x)h^{N_m}_t(x,y)\bar{c}(y)\| &\le C'\|c(x)k_t(x,y)\bar{c}(y)-c(x)h^{N_m}_t(x,y)\bar{c}(y)\|_{\frac{n}{2}+1}
\\& \le C'Ct^m. 
\end{align*}
In fact, since $c(x)$ and $\bar{c}(x_0)$ are compactly supported, the function in the norm is supported in a compact set in $M\times M$, where the theorem can be applied.
\end{proof}

\begin{remark}\label{addtr}
From (\ref{AE}) $\displaystyle\lim_{t\to0+}c(x)\str k_t(x,x)=\lim_{t\to0+}c(x)\frac{1}{(4\pi t)^{n/2}}\sum_{j=0}^{l}t^j \str a_j(x)$ for sufficiently large $l$. To calculate the left hand side it is sufficient to investigate $a_j$s on the right hand side.
\end{remark}

If $a\in\Hom V_x$ then $a$ has a decomposition $a=b\otimes c, b\in\mathrm{\C l}(T_x^{\ast}M), c\in\Hom W$ as in $(\ref{decomposition}).$
The super-trace $\str a$ is then calculated by
$\str(b\otimes c)=\tau(b)\cdot\str^{V/S}(c)$
where $\str^{V/S}$ is the super-trace on $\C$-linear endomorphisms of $W$ under the identification $\Hom_{\C l(T_x^{\ast}M)}(V_x)=\Hom_{\C}(W)$
and $\tau_s$ is the the super-trace on $\Hom S=S\otimes S^{\ast}=\mathrm{\C l}(T_x^{\ast}M)$. 
The super-trace $\tau_s$ on $\mathrm{\C l}(T_x^{\ast}M)$ is explicitly calculated by \cite{BGV} Proposition 3.21.
Let $c=\sum c_{i_1i_2\cdots i_k}e^{i_1}e^{i_2}\cdots e^{i_k}$ be an element in $\mathrm{\C l}(T_x^{\ast}M)=\Hom(S)$, where $c_{i_1i_2\cdots i_k}, 1\le i_1\le i_2\le\cdots\le i_k\le n$ is the coefficient of the element $e^{i_1}e^{i_2}\cdots e^{i_k}$ in $\mathrm{\C l}(T_x^{\ast}M).$ Then 
\begin{equation}\tau_s(c)=(-2i)^{\frac{n}{2}}c_{12\cdots n}.\label{3}\end{equation}
The Clifford algebra $\mathrm{\C l}(T_x^{\ast}M)$ is a filtered algebra, more specifically, $\mathrm{\C l}(T_x^{\ast}M)=\mathrm{\C l}(\R^n)=\cup_{i=0}^n\mathrm{\C l}_i$. 
Here $\mathrm{\C l}_i$ is the linear combination of $e^{j_1}\cdots e^{j_k}, k\le i$.
In proving Theorem \ref{later1}, the following lemma is obtained as a corollary.

\begin{lemma}\label{imlem}
Let $a_i(x)$ be the $i$th term in the asymptotic expansion. 
Then 
\begin{equation}a_i(x)\in\mathrm{\C l}_{2i}\otimes\Hom_{\C l(T_x^{\ast}M)}(S_x).\label{4}\end{equation}  
\end{lemma}

\begin{proof}
 We define $a_i(x)=a_i(x,x)$ to be $\alpha(x,x).$ We need to show that $\alpha_i({y},y)\in\mathrm{\C l}_{2i}\otimes \Hom_{\mathrm{\C l}}(V_{y})$. Set $x=y$ in (\ref{hoho}), then 
$$\alpha_0({y},y)=1\text{ and }\alpha_j({y},y)=-\frac1j(\D^2\alpha_{j-1})({y},y).$$ 
with $\alpha_0({y},y)=1\in\mathrm{\C l}_{0}\otimes \Hom_{\mathrm{\C l}}(V_{y})$. 
Inductively, the fact that $\D^2$ contains the factor $\cl(e_i)\cl(e_j)$, makes sure that the degree of $\alpha_i(x)$ does not increase by more than $2$ compared to that of $\alpha_{i-1}(x).$ 
\end{proof}

\begin{remark}
As a consequence of (\ref{3}) and (\ref{4}) we have $\str a_i(x)=0$ for $i\le\frac{n}{2}.$
Therefore $\displaystyle \ind\D=\frac{1}{(4\pi t)^{\frac{n}{2}}}\sum_{i\ge\frac{n}{2}}t^i\int_Mc(x)\str(a_i(x))\dif x.$
Furthermore, since the index is independent of $t$ and $n$ is even, we have the following theorem.
\end{remark}

\begin{theorem}\label{imthm}
The index of the graded Dirac operator $\D$ is equal to 
\begin{equation}\ind\D=\frac{1}{(4\pi)^{\frac{n}{2}}}\int_Mc(x)\str(a_{n/2}(x))\dif x.\label{inter}\end{equation}
\end{theorem}

The element $\str(a_{\frac{n}{2}}(x))$ in (\ref{inter}) can be calculated analytically in terms of differential forms on $M$. 
To calculate $\str(a_{n/2}(x))\in\Hom(V_x)$, we localize the operator $\D$ and the heat kernel $k_t(x,y)$ at a point $x$.
Because the local calculation is irrelevant to $M$ being compact or not, 
 we use the classical calculation of $\str a_{\frac{n}{2}}$ on a compact manifold without modification.Therefore, $\str(a_{\frac{n}{2}}(x))$ is the $n$ form part of 
$\displaystyle{\det}^{\frac12}(\frac{R/2}{\sinh R/2})\tr^{V/S}(e^{-F}).$ 
For details, please refer to \cite{BGV} Chapter 4.
Finally we obtain the following main theorem of this subsection. 

\begin{theorem}\label{formula!}
Let $R$ be the curvature 2-form with respect to the Levi-Civita connection on the manifold (on $TM$).
Then,
\begin{equation} \ind\D=\int_Mc(x)\hat{A}(M)\cdot \mathrm{ch}(V/S).\label{formula}
\end{equation} 
where $\displaystyle\hat{A}(M)={\det}^{\frac12}(\frac{R/4\pi i}{\sinh{R/4\pi i}})$ is the $\hat{A}$-class of $TM$ and 
$\displaystyle\mathrm{ch}(V/S)=\tr^{V/S}(e^{-F^{V/S}})$ is the relative Chern character, i.e. Chern character of the twisted curvature $F^{V/S}$ of the bundle $S$.
\end{theorem}
\vspace{.3cm}

\subsection{Conclusion.}\label{4.3}

In this subsection we will figure out $\ind D_{V(\sigma_P)}$ where $D$ is the Dolbeault operator on $\Sigma X$, and where $V(\sigma_P)$ is a bundle over $\Sigma X.$ $D_{V(\sigma_P)}$ is a generalized Dirac operator and we calculate the case when $\D=D_{V(\sigma_P)}$, $M=\Sigma X$ in the previous subsections.
Firstly we have the following proposition, as a corollary to Theorem \ref{formula!}.

\begin{proposition}\label{key}
Let $G$ be a locally compact unimodular group and let $M$ be proper cocompact $G$-manifold of dimension $n$ having an almost complex structure, curvature $R$, a cutoff function $c\in C_c^{\infty}(M)$ and a $G$-bundle $E$ with curvature $F$. Let $D: L^2(M, \Lambda^{0,\ast}T^{\ast}M)\rightarrow L^2(M, \Lambda^{0,\ast}T^{\ast}M)$ be the Dolbeault operator on $M$. Then the $L^2$-index of the twisted Dirac operator $D_E$ is, 
$$\ind D_E=\int_Mc\mathrm{Td}(M)\mathrm{ch}(E),$$
where $\mathrm{Td}(M)=\mathrm{det}(\frac{R}{1-e^R})$ and 
$\mathrm{ch}(E)=\mathrm{tr}_s(e^{-F})$. 
\end{proposition}

Both $\mathrm{Td}(M)$ and $\mathrm{ch}(E)$ are $G$-invariant forms. So the integral does not depend on the choice of the cutoff function. If $M=\Sigma X$, then the cutoff function on $M$ can be obtained from the cutoff function on $X$ by setting the values of the elements in the same fiber to be the same. 
The following index formula is immediate assuming the proposition.

\begin{theorem}\label{most important}
Let $X$ be a complete Riemannian manifold where a locally compact unimodular group $G$ acts properly, cocompactly and isometrically. If $P$ is a zero order properly supported elliptic pseudo-differential operator, then the $L^2$ index of $P$ is given by the formula 
\begin{equation}\label{index formula} \ind P=\int_{TX}c(x)(\hat{A}(X))^2 \mathrm{ch}(\sigma_P).\end{equation}
\end{theorem}

\begin{proof}
Set $M=\Sigma X, V=V(\sigma_P)$. Clearly, $M$ has an almost complex structure. By Proposition \ref{ADE} and Proposition \ref{key}, 
$$\ind P=\int_{\Sigma X}c(x)\mathrm{Td}(\Sigma M) \mathrm{ch}(V_{\sigma_P})=\int_{TX}c(x)\mathrm{Td}(TX\otimes\C)\mathrm{ch}(\sigma_P).$$ 
Observe that $\mathrm{Td}(TX\otimes\C)=(\hat{A}(X))^2,$ then the statement follows.
\end{proof}

\begin{proof}[Proof of Proposition \ref{key}]
Let $J$ be an almost complex structure on $M$. Say $x_i, y_i, 1\le i\le m$ are a local frame of $TM$ and $J(x_i)=y_i, J(y_i)=-x_i.$
$J$ extends $\C$-linearly to $TM\otimes\C=TM^{1,0}\oplus TM^{0,1}$ where $\displaystyle TM^{1,0}=\{v-iJv|v \in TM\}$ is the set of holomorphic tangent vectors of the form $z_i\doteq x_j-iy_j$
and $\displaystyle TM^{0,1}= \{v+iJv, v \in TM\}$ is the set of anti-holomorphic tangent vectors of form $\bar{z_j}\doteq x_j+iy_j$.
We have real isomorphisms 
$\displaystyle\pi^{1,0}: TM\rightarrow TM^{1,0}, v\mapsto v^{1,0}=\frac12(v-iJv)$ and $\displaystyle \pi^{0,1}: TM\rightarrow TM^{0,1}, v\mapsto v^{0,1}=\frac12(v+iJv).$
Therefore $(TM,J)\simeq TM^{1,0}\simeq\overline{TM^{0,1}}$ as an almost complex bundle.

Similarly, the complexified cotangent bundle decomposes as $T^{\ast}M\otimes\C=T^{\ast}M^{1,0}\oplus T^{\ast}M^{1,0}$ where 
$\displaystyle T^{\ast}M^{1,0}=\{\eta\in T^{\ast}M\otimes\C|\eta(Jv)=i\eta(v)\},$ consisting of covectors of form $z^j\doteq x^j+iy^j$, is the $\C$-dual of $TM^{1,0}$ (notation: $x^j(x_i)=\delta_{ij}, y^j(y_i)=\delta_{ij}$) and 
$\displaystyle T^{\ast}M^{0,1}=\{\eta\in T^{\ast}M\otimes \C|\eta(Jv)= -i\eta(v)\},$ consisting of covectors of form $\bar{z^j}\doteq x^j-iy^j$, 
is the $\C$-dual of $TM^{0,1}$. 

 Let $\Omega^{\ast}M$ be the set of smooth sections of $\Lambda^{\ast}M$, which splits into types $(p,q)$ with 
 $\Lambda^{p,q}T^{\ast}M=(\Lambda^pT^{\ast}M^{1,0})\otimes(\Lambda^qT^{\ast}M^{0,1}).$
 If $\alpha\in\Omega^{p,q}(M)$, then the differential decomposes into $\displaystyle\dif\alpha=\sum_{i=0}^{p+q+1}(\dif\alpha)^{i,p+q+1-i}$ and set $\partial\alpha=(\dif\alpha)^{p+1,q}, \bar{\partial}\alpha=(\dif\alpha)^{p,q+1}.$ 
The Dolbeault operator $\bar{\partial}: \Omega^{0,q}\rightarrow\Omega^{0, q+1}$ is the order $1$ differential operator given by $\displaystyle\bar{\partial}=\frac{\partial}{\partial y}+i\frac{\partial}{\partial x}$ in the local coordinate $(x, y)\in M$. If we use the grading, the Dolbeault operator is $\displaystyle\bar{\partial}+\bar{\partial}^{\ast}$ on $\Omega^{0,\ast}M$.

The Dolbeault operator ``is" the canonical Dirac operator on $M$ in the sense that they have the same symbol. The canonical Dirac operator on $M$ is defined as follows. The bundle $S=\Lambda^{0,\ast}T^{\ast}M$ has an action of the cotangent vectors via Clifford multiplication: 
$$\cl(\eta)s=\sqrt2(\epsilon(\eta^{0,1})(s)-\iota(\eta^{1,0})s), \eta\in T^{\ast}M, s\in \Lambda^{0,\ast}T^{\ast}M.$$
Here, $\cl(x^i)=\frac{1}{\sqrt2}(\epsilon(\bar{z})-\iota(z)), \cl(x^i)\cl(x^j)+\cl(x^j)\cl(x^i)=-2\delta_{ij}$ and $\epsilon$ is the exterior multiplication and $\iota$ is the $\C$-linear compression by a vector.

The canonical Dirac operator is defined to be $D=\sum \cl(e^i)\nabla^L_{e_i}$ where $\{e_i\}$ forms a local orthonormal basis of $TM$ and $\nabla^L$ is the Levi-Civita connection on $S$.
Now if there is an auxiliary complex $G$-vector bundle $E\rightarrow M$, with a $G$-invariant Hermitian metric and $G$-invariant connection $\nabla^E$, the Dolbeault operator $D_E$ acting on $V=S\otimes E$ with coefficients in $E$ can be represented by (up to a lower order term): 
$$D_E=\sum \cl(e_i)\nabla_{e_i}^V,\text{ where }\nabla^V=\nabla^L\otimes1+1\otimes\nabla^E.$$

Let $\nabla$ be the Levi-Civita connection on $M$ (on $(TM)^{0,1}$, being more precise) and $R=\nabla^2\in\Lambda^2(M, \mathfrak{so}(TM))$ be Riemannian curvature, the matrix with coefficients of two forms representing the curvature of $M$, 
$$R(X,Y)=\nabla_X\nabla_Y-\nabla_Y\nabla_X-\nabla_{[X,Y]}, X, Y\in C^{\infty}(M, TM).$$
In the orthonormal frame $e_i$ of $TM$, $\displaystyle R(e_i, e_j)=-\sum_{k<l}(R(e_i, e_j)e_l,e_k)e^k\wedge e^l,$
where we identify $\mathfrak{so}(TM)$ with the bundle of two forms on $M.$
Now we have a Clifford module $S$, where ${\C}l(T^{\ast}M)=\Hom(S)$, on which $T^{\ast}M$ acts by Clifford multiplication.
On $S$ there is a Clifford connection $\nabla^S$ so that the Clifford multiplication by unit vectors preserves the metric and $\nabla^S$ is compatible with the connection on $M.$
Let $R^S=(\nabla^S)^2$ be the curvature associated to $\nabla^S.$
It is well known that the Lie algebra isomorphism $\mathfrak{spin}_n\simeq \mathfrak{so}_n$ given by $\frac14[v,w]\mapsto v\wedge w$ implies that 
$$R^S(e_i, e_j)=\frac12\sum_{k<l}(R(e_i, e_j)e_k, e_l)\cl(e_k)\cl(e_l)=\frac14\sum_{k,l}(R(e_i,e_j)e_k,e_l)\cl(e_k)\cl(e_l).$$

On $S$, there is also a Levi-Civita connection, denoted by $\nabla^L$. The associated curvature $R^L=(\nabla^L)^2\in\Hom S$ is written as 
$$R^L=R^S+F$$ 
where $\displaystyle R^S(\cdot,\cdot)=\frac14\sum_{k,l}(R(\cdot,\cdot)\bar{z_k},z_l)\cl(\bar{z_k})\cl(z_l)+\frac14\sum_{k,l}(R(\cdot,\cdot)z_k,\bar{z_l})\cl(z_k)\cl(\bar{z_l})\in \mathrm{Cl}(TM)$ and $F\in\Hom_{Cl}V$ is the twisting curvature.

Recall that the curvature of the Levi-Civita connection on $\Lambda V^{\ast}$ is the derivation of the algebra $\Lambda V^{\ast}$ which coincides with $R(e_i, e_j)$ on $V$ and is given by the formula 
$$\sum_{k,l}<e^k, R(e_i, e_j)e_l>\epsilon(e_k)\iota(e^l)=\sum_{k,l}(R(e_i, e_j)e_l, e_k)\epsilon(e^k)\iota(e^l).$$
Let $R^-$ be the curvature of the Levi-Civita connection on $T^{0,1}M$. Note that $R=R^-$. 
Then the curvature of $\nabla^L$ on $S$ is given by 
$$R^L(\cdot,\cdot)=\frac14\sum_{i,j}(R^-(\cdot,\cdot)z_i,\bar{z_j}) \epsilon (\bar{z_j})\iota(z_i)=-\frac18\sum_{i,j}(R^-(\cdot,\cdot)z_i,\bar{z_j})\cl(\bar{z_j})\cl(z_i).$$

Using the fact that $\cl(z_i)^2=0, \cl(\bar{z_i})^2=0, \cl(z_i)\cl(\bar{z_j})+\cl(\bar{z_j})\cl(z_i)=-4\delta_{ij}$, where $\cl(z)=\cl(x)+i\cl(y), \cl(\bar{z})=\cl(x)-i\cl(y).$
we have $$F^{V/S}=R^V-R^S=\frac12\sum_k(Rz_k,\bar{z_k})=\frac12\Tr R+F^E$$ and a direct calculation shows that

$$\hat{A}(M)e^{F^{V/S}}=\mathrm{det}\frac{R/2}{\sinh R/2}e^{\frac12\Tr R}(e^{F^E})=\mathrm{det}\frac{R}{e^R-1}(e^{F^E})=\mathrm{Td}(M)\Tr(e^{-F^E}).$$
\end{proof}

The following theorem is an immediate corollary to Theorem \ref{most important}.

\begin{theorem}[Atiyah's $L^2$-index theorem]
Let $D$ be an elliptic operator on a compact manifold $X$ and $\tilde{D}$ be the $\pi_1(X)$-invariant operator defined on the universal cover space $\tilde{X}$ as the lift of $D$. Then $\ind\tilde{D}=\ind D.$
\end{theorem}

\subsection{$L^2$-index theorem for homogeneous spaces of Lie groups.}

Let $G$ be a unimodular Lie group and $H$ be a compact subgroup. Consider the homogenous space $M=G/H$ of left cosets of $H$ in $G$, a $G$-bundle $\bar{E}$ over $M$ and a $G$-invariant elliptic operator $D$ on $\bar{E}$. The fiber of $\bar{E}$ at $eH$, denoted by $E=\bar{E}|_{eH}$, is an $H$-space, so that $\bar{E}=G\times_H E.$
Similarly, set $V=T_{eH}M,$ then $TM=G\times_H V.$ 
Let $\Omega\in\Lambda^2(TM)^{\ast}\otimes \mathfrak{gl}(TM)$ be the curvature of $M$, associated to the $G$-invariant Levi-Civita connection on $TM.$
Then we have the $G$-invariant $\hat{A}$-class 
$$\hat{A}(M)=\mathrm{det}^{\frac12}\frac{\Omega/4\pi i}{\sinh\Omega/4\pi i}.$$   
Let $\Omega^E\in\Lambda^2 (\Sigma M)^{\ast}\otimes \mathfrak{gl}(V(\sigma_A))$ be a curvature form associated to some $G$-invariant connection on $V(\sigma_D)$ over $\Sigma M$.
Then 
$$\mathrm{ch}(\sigma_D)=\Tr e^{\Omega^E}|_{TM}$$ 
is the Chern character of $V(\sigma_A)$ restricted to $TM.$
Let $\Omega_V$ be the curvature tensor $\Omega$ restricted to $V=T_{eH}M$ and $\Omega^E_V$ be the curvature tensor $\Omega^E$ restricted to $V.$ 
Then we define the corresponding $\hat{A}$-class and Chern character as
$$\hat{A}(M)_V\doteq\mathrm{det}^{\frac12}\frac{\Omega_V/2}{\sinh\Omega_V/2}\text{ and }\mathrm{ch}(\sigma_D)_V\doteq\Tr e^{\Omega_V^E}.$$
  We have as a corollary the $L^2$-index theorem for homogeneous spaces.

\begin{corollary}
The $L^2$-index of a $G$-invariant elliptic operator $D: L^2(M,\bar{E})\rightarrow L^2(M, \bar{E})$ is 
\begin{equation}\ind D=\int_V\hat{A}^2(M)_V\mathrm{ch}(\sigma_D)_V.\label{CMformula}\end{equation}
\end{corollary}

\begin{proof}
The $L^2$-index theorem of $D$ says that $$\ind D=\int_{TM} c\hat{A}^2(M)\mathrm{ch}(\sigma_D).$$ 
Since $TM=G\times_H V$, the integration of the form $c\hat{A}^2(M)\mathrm{ch}(\sigma_A)$ on $TM$ can be computed by lifting to an $H$-invariant form on $G\times V$ and then integrating over the group part and then the tangent space at $eH.$
Since $\hat{A}^2(M)\mathrm{ch}(\sigma_D)$ is $G$-invariant, then at any $g\in G$, the form will be the same as its value at the unit $e$ of $G$: $\hat{A}^2(M)_V\mathrm{ch}(\sigma_A)_V.$
Hence, 
$$\int_{TM} c\hat{A}^2(M)\mathrm{ch}(\sigma_D)=\int_V \hat{A}^2(M)_V\mathrm{ch}(\sigma_D)_V\int_Gc(g^{-1}v)\mathrm{vol}=\int_V\hat{A}^2(M)_V\mathrm{ch}(\sigma_D)_V,$$ 
where $\mathrm{vol}$ is the volume form on $G.$
\end{proof}

\begin{remark}
The formula (\ref{CMformula}) is essentially the $L^2$-index formula in \cite{Connes:1982}. 
The components of the formula in \ref{CMformula} are sketched as follows.
On the Lie algebra $\mathfrak{g}$ of $G$ there is an $H$-invariant splitting $\mathfrak{g}=\mathfrak{h}\oplus\mathfrak{m}$ where $\mathfrak{h}$ is the Lie algebra of $H$ and $\mathfrak{m}$ is an $H$-invariant complement. $V=T_{eH}(G/H)$ is a candidate for $\mathfrak{m}.$ There is a curvature form on $m$ defined by $\Theta(X, Y)=-\frac12\theta([X,Y]), X,Y\in \mathfrak{m} \label{curvature form}$
 where $\theta$ is the connection form given by the projection $\theta: \mathfrak{g}\rightarrow\mathfrak{h}.$ 
$\Theta$ composed with $r: \mathfrak{h}\rightarrow \mathfrak{gl}(E)$, the differential of a unitary representation of $H$ on some vector space $E$, is an $H$-invariant curvature form 
$\Theta_r(X,Y)=r(\Theta(X,Y)), X, Y\in\mathfrak{m}.$
Then 
$$\mathrm{ch}: R(H)\rightarrow H^{\ast}(g,H): r\mapsto \Tr e^{\Theta_r}$$ 
is a well-defined Chern character  (\cite{Connes:1982} page 309). 
Also, compose the curvature form (\ref{curvature form}), with $\mathfrak{h}\rightarrow \mathfrak{gl}(V)$, the differential of the $H$-module structure of $V$. And a curvature form $\Theta_V\in\Lambda^2\mathfrak{m}^{\ast}\otimes \mathfrak{gl}(V)$ on $V$ is constructed and the $\hat{A}$-class is defined as 
$$\hat{A}(\mathfrak{g},H)=\mathrm{det}^{\frac12}\frac{\Theta_V/2}{\sinh\Theta_V/2}.$$
The $L^2$-index formula of $D$ in \cite{Connes:1982} is 
\begin{equation}\ind D=\int_V\mathrm{ch}(a)\hat{A}(\mathfrak{g},H),\label{CM formula}\end{equation}
where $a$ is an element of the representation ring $R(H)$, specifically $a$ is the pre-image of $V(\sigma_D)|_{V^+}$ under the Thom isomorphism $R(H)\rightarrow K_H(V).$ Here, $V^+$ is the space built from $V$ by adding one point at infinity. It is the ball fiber in $\Sigma M$ at $eH.$ Note that the Thom isomorphism exists only for the case when the action of $H$ on $V$, lifts to $\mathrm{Spin}(V).$ The general case was done by introducing a double covering of $H$ and by reducing the problem to this situation \cite{Connes:1982} page 307.

To see that \ref{CMformula} and \ref{CM formula} are the same formula, we prove the following assertions.

(1) $\hat{A}(M)_V=\hat{A}(\mathfrak{g},H)$.
     
     In fact, since $TM=G\times_HV$ is a principal $G$-bundle over $V/H$ and $V$ is a principal $H$-bundle over $V/H$, then by \cite{KNI} II Prop. 6.4, the connection form on $TM$ restricted to $V$ is also a connection form.  
     Also, on $G/H$, the restriction of any $G$-invariant tensor on $TM$ to $V$ is an $H$-invariant  tensor on $V.$ 
     Therefore $\Omega_V$ is an $H$-invariant curvature form on $V$ and the restriction $\hat{A}(M)_V$ is the $\hat{A}$-class defined by curvature $\Omega_V.$ By definition $\hat{A}(\mathfrak{g},H)$ is the $\hat{A}$-class of the curvature $\Theta_V$ on $V$, $\hat{A}$-class of another connection on the same $V$. The statement is proved because $\hat{A}$ is a topological invariant and is independent of the choice of connection on $V.$

(2)  $\mathrm{ch}(\sigma_D)_V=\mathrm{ch}(a)$.
   
   Similarly to the last proof, $\Omega_V^E$ is an $H$-invariant curvature form of $V(\sigma_D)|_{V^+}$ restricted to $V.$
     Recall that $V(\sigma_D)$ is glued by the $G$-invariant symbol $\sigma_D$ and therefore it is determined by its restriction at the ball fiber, $V^+$. By definition $V(\sigma_D)|_{V^+}$ is glued two copies of $BV\times E$ on the boundary by $\sigma_D|_{SV}.$ Note that the evaluation of $\sigma_D|_{SV}$ at $\xi\in SV$ is 
     $\sigma_D(eH, \xi)\in GL(E), \xi\in V, \|\xi\|=1.$
    We have an $H$-bundle $V(\sigma_D)|_{V}=V\times_H E$ where $r: H\rightarrow E.$ Hence the curvature $\Omega_V^E$ is $r$ composed with some curvature form on $V.$ 
     The statement follows from the fact that $\mathrm{ch}(r)$ is independent of the connection and the choice of the $H$-invariant splitting of $G.$
    
\end{remark}
\bibliographystyle{abbrv}
\bibliography{Generalbib}

\begin{thebibliography}{10}

\bibitem{Antonyan:2010}
S.~A. Antonyan.
\newblock Proper actions on topological groups: Applications to quotient
  spaces.
\newblock {\em Proceedings of the American Mathematical Society},
  138:3707--3716, 2010.

\bibitem{Atiyah:1976}
M.~Atiyah.
\newblock Elliptic operators, discrete groups and von {N}eumann algebras.
\newblock {\em Soci\'{e}t\'{e} Math\'{e}matique de France}, 1976.

\bibitem{ABP:1973}
M.~Atiyah, R.~Bott, and V.~K. Patodi.
\newblock On the heat equation and the index theorem.
\newblock {\em Inventiones Mathematicae}, (19):279--330, 1973.

\bibitem{Atiyah:1968pd}
M.~Atiyah and I.~Singer.
\newblock The index of elliptic operators {III}.
\newblock {\em The Annals of Mathematics}, 87(3):546--604, 1968.

\bibitem{BD:1982}
P.~Baum and R.~G. Douglas.
\newblock K homology and index theory.
\newblock {\em Proc. Symp. Pure Math}, 38(1):117--173, 1982.

\bibitem{BGV}
N.~Berline, E.~Getzler, and M.~Vergne.
\newblock {\em Heat kernels and Dirac operators}.
\newblock Springer, 2003.

\bibitem{Bredon}
G.~Bredon.
\newblock {\em Introduction to compact transformation groups}.
\newblock Academic Press, 1972.

\bibitem{Breuer:1968}
M.~Breuer.
\newblock Fredholm theories in von {N}eumann algebras. {I}.
\newblock {\em Math. Ann.}, 178:243--254, 1968.

\bibitem{Connes:1994zh}
A.~Connes.
\newblock {\em Noncommutative Geometry}.
\newblock Academic Press, 1994.

\bibitem{Connes:1982}
A.~Connes and H.~Moscovici.
\newblock The ${L}^2$-index theorem for the homogeneous spaces of {L}ie groups.
\newblock {\em Ann. of Math.}, 115:291--330, 1982.

\bibitem{Connes:1990ht}
A.~Connes and H.~Moscovici.
\newblock Cyclic cohomology, the {N}ovikov conjecture and hyperbolic groups.
\newblock {\em Topology}, 29(3):345--388, 1990.

\bibitem{Farsi:1992}
C.~Farsi.
\newblock {K}-theoretical index theorems for good orbifolds.
\newblock {\em Proceedings of the American Mathematical Society},
  115(3):769--773, Jul. 1992.

\bibitem{Getzler:1986}
E.~Getzler.
\newblock A short proof of the local {A}tiyah-{S}inger index theorem.
\newblock {\em Topology}, pages 111--117, 1986.

\bibitem{Gilkey:1973}
P.~B. Gilkey.
\newblock Curvature and eigenvalues of the {L}aplacian for elliptic complexes.
\newblock {\em Advances in Mathematics}, pages 433--382, 1973.

\bibitem{Gilkey:1974}
P.~B. Gilkey.
\newblock {\em The index theorem and the heat equation}.
\newblock Publish or Perish, inc. (Boston), 1974.

\bibitem{Kasparov:1981}
G.~Kasparov.
\newblock The operator ${K}$-functor and extensions of ${C}^{\ast}$-algebras.
\newblock {\em Mathematics of the USSR - Isvestiya}, 16(3):513--572, 1981.

\bibitem{Kasparov:1983}
G.~Kasparov.
\newblock The index of invariant elliptic operators, {K}-theory, and {L}ie
  group representations.
\newblock {\em English translation: Soviet Mathematics-Doklady,}, 27:105--109,
  1983.

\bibitem{Kasparov:1988dw}
G.~Kasparov.
\newblock Equivariant ${KK}$-theory and the {N}ovikov conjecture.
\newblock {\em Inventiones Mathematicae}, 91(1):147--201, 1988.

\bibitem{Kasparov:2008}
G.~Kasparov.
\newblock K-theoretic index theorems for elliptic {K}-theoretic index theorems
  for elliptic and transversally elliptic operators.
\newblock {\em Preprint}, 2012.

\bibitem{KNI}
S.~Kobayashi and K.~Nomizu.
\newblock {\em Foundations of Differential Geometry}, volume~I.
\newblock Weily-Interscience, weily classics library edition edition, 1996.

\bibitem{Lesch:2009}
M.~Lesch, H.~Moscovici, and M.~J. Pflaum.
\newblock {C}onnes-{C}hern character for manifolds with boundary and eta
  cochains.
\newblock {\em arXiv:0912.0194v2}, 2009.

\bibitem{Lott:1992}
J.~Lott.
\newblock Superconnections and higher index theory.
\newblock {\em Geometric and Functional Analysis}, 2(4):421--454, 1992.

\bibitem{Mathai:2001}
M.~Marcolli and V.~Mathai.
\newblock Twisted index theory on good orbifolds, {II}: Fractional quantum
  numbers.
\newblock {\em Communications in Mathematical Physics}, 217:55--87, 2001.

\bibitem{Mathai:10}
V.~Mathai and W.~Zhang.
\newblock Geometric quantization for proper actions.
\newblock {\em Advances in Mathematics}, (3):1224--1247, 2010.

\bibitem{Miscenko:1980fu}
A.~S. Mi\v{s}\v{c}enko and A.~T. Fomenko.
\newblock The index of elliptic operators over ${C}^*$-algebras.
\newblock {\em Mathematics of the USSR - Isvestiya}, 15(1):87--112, 1980.

\bibitem{Perez:2008}
J.~J. Perez.
\newblock The $\mathrm{G}$-fredholm property of the $\bar{\partial}$-{N}eumann
  problems.
\newblock {\em J. Geom. Anal.}, 2008.

\bibitem{Perrot:2009}
D.~Perrot.
\newblock The equivariant index theorem in entire cyclic cohomology.
\newblock {\em Journal of $\mathrm{K}$-theory}, 3(2):261--307, 2009.

\bibitem{Roe:1988qy}
J.~Roe.
\newblock An index theorem on open mainfolds, {I}.
\newblock {\em Journal of differential geometry}, 27:87--113, 1988.

\bibitem{Roe:1988uq}
J.~Roe.
\newblock An index theorem on open mainfolds, {II}.
\newblock {\em Journal of differential geometry}, 27:115--136, 1988.

\bibitem{Roe:1998ad}
J.~Roe.
\newblock {\em Elliptic Operators, topology and asymptotic methods}.
\newblock Chapman and Hall, second edition, 1998.

\bibitem{Schick:2005}
T.~Schick.
\newblock $\mathrm{L}^2$-index theorems, $\mathrm{KK}$-theoory, and
  connections.
\newblock {\em New York Journal of Mathematics}, 11:387--443, 2005.

\bibitem{Shubin}
M.~Shubin.
\newblock Von {N}eumann algebra and $\mathrm{L}^2$ techniques in geometry and
  topology.

\bibitem{Skandalis:1984}
G.~Skandalis.
\newblock Some remarks on {K}asparov theory.
\newblock {\em Journal of functional analysis}, 56(3):337--347, 1984.

\bibitem{Stern:1989}
M.~Stern.
\newblock ${L}^2$-index theorems on locally symmetric spaces.
\newblock {\em Inventiones Mathematicae}, 96:231--282, 1989.

\end{thebibliography}

\end{document}